\documentclass[reqno,11pt]{amsart}
\usepackage{epsf}
\usepackage{mhequ}
\usepackage{amssymb}
\usepackage{graphics}
\usepackage{graphicx}
\usepackage{wrapfig}

\def\eps{\varepsilon}

\def\be{\begin{equation}}
\def\ee{\end{equation}}
\def\ba{\begin{align}}
\def\bm{\begin{multline}}
\def\bfig{\begin{figure}[htb]}
\def\efig{\end{figure}}

\numberwithin{equation}{section}
\newtheorem{theorem}{Theorem}[section]
\newtheorem{proposition}[theorem]{Proposition}
\newtheorem{lemma}[theorem]{Lemma}

\newtheorem{remark}[theorem]{Remark}

\DeclareMathSymbol{\leqslant}{\mathalpha}{AMSa}{"36}
\DeclareMathSymbol{\geqslant}{\mathalpha}{AMSa}{"3E}
\DeclareMathSymbol{\doteqdot}{\mathalpha}{AMSa}{"2B}
\DeclareMathSymbol{\circlearrowright}{\mathalpha}{AMSa}{"08}
\DeclareMathSymbol{\subsetneq}{\mathalpha}{AMSb}{"28}
\DeclareMathSymbol{\supsetneq}{\mathalpha}{AMSb}{"29}
\renewcommand{\leq}{\;\leqslant\;}
\renewcommand{\geq}{\;\geqslant\;}
\newcommand{\isdefby}{\; := \;}
\newcommand{\bydefis}{\; =: \;}

\newcommand{\dd}{{\rm d}}
\newcommand{\e}[1]{\,{\rm e}^{#1}\,}

\newcommand{\upchi}{\raise 2pt \hbox{$\chi$}}

\def\writefig#1 #2 #3 {\rlap{\kern #1 truecm \raise #2 truecm
\hbox{#3}}}

\newcommand{\caA}{{\mathcal A}}
\newcommand{\caB}{{\mathcal B}}

\newcommand{\caD}{{\mathcal D}}

\newcommand{\caF}{{\mathcal F}}

\newcommand{\caK}{{\mathcal K}}

\newcommand{\caO}{{\mathcal O}}

\newcommand{\caS}{{\mathcal S}}

\newcommand{\caV}{{\mathcal V}}

\newcommand{\caX}{{\mathcal X}}

\newcommand{\bbE}{{\mathbb E}}

\newcommand{\bbP}{{\mathbb P}}
\newcommand{\bbQ}{{\mathbb Q}}
\newcommand{\bbR}{{\mathbb R}}

\newcommand{\Ai}{\mathrm{Ai}}
\newcommand{\Bi}{\mathrm{Bi}}
\newcommand{\ve}{\varepsilon}

\newcommand{\tq}{\tilde{q}}
\newcommand{\Var}{\text{Var}}
\newcommand{\deter}{\text{det}}

\newcommand{\sez}{\sigma \downarrow 0}
\newcommand{\vez}{\varepsilon \downarrow 0}

\newcommand{\si}{\sigma}

\begin{document}


\title{A chain of interacting particles under strain}

\author{Michael Allman, Volker Betz and Martin Hairer}
\address{Michael Allman, Volker Betz, Martin Hairer \hfill\newline
\indent Department of Mathematics \hfill\newline
\indent University of Warwick \hfill\newline
\indent Coventry, CV4 7AL, England \hfill\newline}
\email{m.j.allman@warwick.ac.uk, v.m.betz@warwick.ac.uk, m.hairer@warwick.ac.uk}

\maketitle

\begin{quote}
{\small
{\bf Abstract.}
We investigate the behaviour of a chain of interacting Brownian particles with one end fixed and the other moving away at slow speed $\ve > 0$, in the limit of small noise. The interaction between particles is through a pairwise potential $U$ with finite range $b>0$. We consider both overdamped and underdamped dynamics.
}  

\vspace{1mm}
\noindent
{\footnotesize {\it Keywords:} Interacting Brownian particles; Pitchfork bifurcation; Stochastic differential equations; Singular perturbation.}

\vspace{1mm}
\noindent
{\footnotesize {\it 2000 Math.\ Subj.\ Class.:} 60J70, 60H10, 34E15.}
\end{quote}


\section{Introduction}\label{sec:intro}

The behaviour of a Brownian particle moving in a potential well and acted upon by a linearly increasing force is widely used to model the mechanical failure of molecular bonds arising in dynamic force spectroscopy experiments \cite{LCST07,DFKU03,REBENMRAR06,Fri08}. This began with the work of Bell \cite{B78} and was developed further by Evans and Ritchie \cite{ER97}.

Let $q_s$ denote the length at time $s$ of a bond that is fixed at one end and has a harmonic spring attached to the other. If the spring moves linearly at speed $\ve > 0$, the motion of $q_s$ is typically modelled according to an SDE of the form
\[
\dd q_s = (-U'(q_s)+\ve s)\, \dd s + \si \, \dd W_s\;,
\]
where $U(q)$ denotes the bond energy (e.g. Lennard-Jones potential), $W_s$ is a standard Brownian motion and $\si > 0$ is the (small) noise intensity. Note that this model assumes the motion is overdamped. Rupture of the bond corresponds to the first time $q_s$ escapes from the stable well of the effective, time-dependent potential, $H(q,\ve s) = U(q) - \ve s\,q$. The effect of the external force is to lower the barrier height of $H$, thus making escape more likely. The main objective is to study the distribution of first-breaking times and how the mean first-breaking time scales with the pulling speed $\ve$. Typically, two pulling speed regimes are considered.

For very slow pulling, the particle is able to escape the well through a large deviation event before the potential has changed significantly and the energy barrier is still large. In order to apply the standard theory valid for time-independent potentials \cite{K40,Eyr35,BEGK04,FW98}, the adiabatic approximation is used: at any given time $s$, the bond has an instantaneous rate of rupture, $k(s)$, and the probability of survival until time $s$, denoted $P(s)$, decays according to $\dot P(s) = -k(s)P(s)$. Note that $1/k(s)$ is the usual Eyring-Kramers formula \cite{K40,Eyr35,BEGK04} applied to $H$ at time $s$.

As the speed of pulling increases, the energy barrier at the time of rupture becomes smaller. If pulling is sufficiently fast, the barrier may be close to vanishing completely when rupture occurs. This means that the external force, given by $\ve s$, is almost equal to the maximum slope of $U$, which occurs at the point of inflection between its minimum and maximum, i.e. maximum slope is $U'(c_0)$, where $U''(c_0) = 0$. For times $s$ at which $\ve s$ is close to this critical force, the effective potential $H$ is almost cubic near its minimum. This leads to a different rupture rate than that given above, although still calculated within the Kramers framework.

It is interesting to consider what happens as the pulling speed increases yet further and the Eyring-Kramers formula is no longer applicable, nor the adiabatic approximation underpinning the above approach. In this paper, we consider this situation in a model related to that above. More precisely, we consider a chain of two identical bonds in series with one end fixed and the other being pulled at a constant rate $\ve$. Both overdamped and underdamped dynamics are treated. We are interested in which of the two bonds breaks first and how this depends on $\ve$ and the noise intensity $\si$. As above, the dynamics near the inflection point of the bond energy $U$ play an important role and will be the focus of our analysis. Roughly, we find that for $\ve > \si^{4/3}$, the right-hand bond breaks first, while for $\ve < \si^{4/3}$, both have an equal probability of breaking in the limit of small noise. Thus $\ve = \si^{4/3}$ represents the threshold at which the adiabatic approximation becomes valid.

To our best knowledge, the first work to tackle rigorously such models of bonds under an external, time-dependent force was \cite{AB09}. There the authors consider a similar model of two bonds in series as above, but with an additional assumption that $U$ is cut-off strictly convex. The breaking event corresponds to the first time one of the two bonds exceeds the range of $U$. Roughly speaking, it is shown that for $\ve > \si$, the chain always breaks on the right-hand side, whereas for $\ve < \si$, each bond has an equal chance to break in the small noise limit. Thus the threshold between the different types of behaviour is different from that found in the present work, where the bond energy $U$ is taken to be smooth (but also with finite range). In principle, the results of \cite{AB09} can be extended to arbitrarily many bonds in series \cite{Allman}.

The behaviour of several bonds in series has also been considered by many authors, for both time-dependent and -independent external forces. The situation when the external force is constant, i.e. one initially stretches the chain by some amount and then fixes both endpoints, has been considered for harmonic potentials \cite{Lee09} and Lennard-Jones potentials \cite{SDG06}. In the harmonic case, it is shown analytically and numerically that the probability to break at either endpoint is half that of breaking at any non-extremal point, which all have the same probability. In the Lennard-Jones case, the motion is not assumed to be overdamped, i.e. the authors consider the equation
\[
m \ddot q_i(s) = -m\gamma \dot q_i(s) - \frac{\partial H}{\partial q_i}(q(s)) + \sqrt{2\gamma} \dot W_i(s)\;,\quad 2 \leq i \leq N-1\;,
\]
where the $W_i$ are independent Brownian motions, $q(s) = (q_i(s))_{i=1}^{N} \in \bbR^N$ is the chain configuration, with $q_1$ and $q_N$ constant, and $H$ is the total potential energy of the chain. It is assumed beforehand that one bond is close to breaking and all others are close to minimal energy so that a quadratic approximation can be used for $H$. Then the breakage rate is calculated, alongside simulations, using a multi-dimensional version of Kramers' theory developed by Langer \cite{Lan69}. These show that the breakage rate is lower the closer the chosen weak bond is to the chain endpoints. However, the simulations also show that the harmonic approximation for $H$ may fail for bonds near the breaking bond, which the authors there suggest may explain some discrepancies between the theory and simulations. As they point out, Langer's theory, as well as the classical Kramers theory for a single particle, requires a harmonic approximation for $H$. 

The case of a chain with one end fixed and a linearly increasing force applied at the other end has been considered by Fugmann and Sokolov in \cite{FS09a,FS09b} to model the mechanical failure of a polymer chain. More precisely, they consider the vector $q(s) = (q_i(s))_{i=0}^N \in \bbR^{N+1}$, with $q_0 \equiv 0$, which evolves according to the SDE
\[
\dd q_i(s) = -\frac{\partial H}{\partial q_i}(q(s),\ve s)\, \dd s + \si \, \dd W_i(s)\;,\quad 1 \leq i \leq N\;,
\]
where $H(q,\ve s)$ is the time-dependent potential energy of the chain, given by
\[
H(q,\ve s) = \sum_{0 \leq i < j \leq N} U(q_i - q_j) - \ve s\, q_N\;,
\]
and $U$ is the Morse potential. A break is said to occur when $q(s)$ overcomes an energy barrier of the effective potential, $H$. Numerically, they show that for a high pulling speed, $\ve$, only the right half of the chain contributes to the breaking event and the probability increases as you move towards the right. For smaller $\ve$, they show that the breakpoint is more uniformly distributed along the whole chain. In their analysis, they assume that the rupture dynamics of different bonds are independent and then apply the one-dimensional theory. The only thing left to do then is to analyse how much force each bond feels, which depends on the pulling speed. For lower pulling speeds, it is assumed that each bond feels the same force, whereas for higher pulling speeds the chain is approximated by a harmonic chain.

This paper is organised as follows. In Section \ref{sec:model}, we introduce our model and deduce equation (\ref{SDEintro}), which is our main object of study. Our results are then stated in Theorems \ref{thm:linear}, \ref{thm:overfull} and \ref{thm:mass sample paths}. In Sections \ref{sec:linear}, \ref{sec:over} and \ref{subsec:sample} we give the proofs.

The following notation is used in this paper:
\begin{itemize}
	\item By $x_1(t,\ve) \asymp x_2(t,\ve)$ we mean that for two functions $x_1(t,\ve)$, $x_2(t,\ve)$, defined for $t$ in an interval $I$ and $0 < \ve \leq \ve_0$, there exist constants $c_{\pm} > 0$ such that
\[
c_- x_2(t,\ve) \leq x_1(t,\ve) \leq c_+ x_2(t,\ve)
\] 
for all $t \in I$ and $0 < \ve \leq \ve_0$.
	\item By $o_x(1)$ we mean that $\lim_{x \to 0}o_x(1) = 0$.
	\item $x_1 \lesssim x_2$ means that there exists a constant $c > 0$ such that $c\, x_1 \leq x_2$ for $x_1, \, x_2 > 0$ sufficiently small. 
	\item $x_1 \ll x_2$ means that $x_1 = o(x_2)$.
	\item We shall write $\bbP^{t_0,q_0}$ to denote probability conditioned on the relevant process starting at time $t_0$ in position $q_0$.
\end{itemize}

{\small {\bfseries Acknowledgements:} M.A. would like to thank Barbara Gentz for many helpful discussions during a visit to Bielefeld University, supported by SFB 701, and the Courant Institute for its hospitality. M.A. is supported by EPSRC Award EP/P502810/1 and a Warwick University ASSEC grant. V.B. is supported by EPSRC Fellowship EP/D07181X/1.
M.H. gratefully acknowledges support by the EPSRC through an Advanced Research Fellowship EP/D071593/1.}

\section{The model and main results}\label{sec:model}

Three particles, $q_L$, $q$ and $q_R$ in $\bbR$, interact with each other via a pairwise potential $U$. We assume that $U$ is smooth with finite range $b > 0$ and a unique minimum at $0 < a < b$, with $U''(a) > 0$. We also assume that there is a unique $c_0 \in (a,b)$ such that $U''(c_0) = 0$. The particle $q_L$ is fixed at the origin and the position of $q_R$ at time $s \geq 0$ is given by $q_R(s)=2a(1+\ve s)$, where $\ve > 0$ is a small parameter.  We study the behaviour of the middle particle, with position at time $s$ given by $q_s$.  Initially, it has position $q_0 = a$ so that the distance between neighbouring particles is $a$. The middle particle evolves according to an SDE of the form
\begin{align*}
\dd q_s & = p_s\, \dd s\;, \\
\ve^{\beta-1}\dd p_s & = -p_s\, \dd s - \frac{\partial H}{\partial q}(p_s,q_s,\ve s)\, \dd s + \si\, \dd W_s \;,
\end{align*}
where $W_s$ is a standard one-dimensional Brownian motion with $W_0 = 0$, $\si > 0$ is the noise intensity, $\beta \in \bbR$ and $H(p,q,\ve s)$ is given by
\[
H(p,q,\ve s) = \frac{p^2}{2} + U(q) + U(2a(1+\ve s) - q)\;.
\]
Rescaling time as $t=\ve s$, this is the same in law as solving
\be\label{eq:SDE before taylor}
\begin{split}
\dd q_t & = \frac{1}{\ve}p_t\, \dd t\;, \\
\ve^{\beta-1}\dd p_t & = -\frac{1}{\ve}p_t\, \dd t - \frac{1}{\ve}\frac{\partial H}{\partial q}(p_t,q_t,t)\, \dd t + \frac{\si}{\sqrt{\ve}}\, \dd W_t \;.
\end{split}
\ee
The length of the chain $(q_L,q,q_R)$ increases linearly with $t$. Clearly, if we wait a long enough time, the distance between $q$ and at least one of its neighbours must become greater than the range of $U$. In this case, these particles no longer interact and the chain can be considered broken. Since $U$ has a minimum at $a$, it is energetically preferable for $q$ to move towards either $q_L$ or $q_R$. Letting $\ve = \ve(\si)$, our aim is to determine how the speed of pulling affects which of these two possibilities occurs in the limit as $\sez$.

We easily check that the configuration of equally spaced particles satisfies $\partial_q H = 0, \,\partial^2_q H > 0$ until time $t_0$, where $a(1+t_0) = c_0$. Thus until this time it is a stable configuration and so we expect $q_{t_0} \approx a(1+t_0)$. For $t>t_0$, this configuration becomes unstable and new minima emerge. So we expect $q_t$ to quickly move away from the chain midpoint and towards one of these newly formed minima. Note that as a function of $q$, $H$ is symmetric about $q = a(1+t)$, but its time-dependence introduces asymmetry as we shall see. Once $q_t$ has approached one of these new minima, we expect it to stay there as the energy barrier to escape becomes higher. The evolution of the chain, therefore, is determined by its behaviour around the bifurcation of $H$ at $t=t_0$, which we shall now consider.

Letting $z_t = a(1+t) - q_t$, we express the term $\partial H/\partial q$ appearing in (\ref{eq:SDE before taylor}) in terms of $z$. By a Taylor expansion in space, we find
\begin{align*}
\frac{\partial H}{\partial q}(p_t,q_t,t) & = U'(q_t)-U'(2a(1+t)-q_t)\\
& = U'(a(1+t)-z_t) - U'(a(1+t) + z_t)\\
& \approx -2U''(a(1+t))z_t - \frac{1}{3}U^{(4)}(a(1+t))z_t^3\;.
\end{align*}
Assuming that there is $0 < T < t_0$ such that for $t \in [t_0 - T,t_0 + T]$, $U^{(4)}(a(1+t))$ is negative and bounded away from zero (see comment below), we have by a Taylor expansion in time,
\[
-2U''(a(1+t))z_t - \frac{1}{3}U^{(4)}(a(1+t))z_t^3 \approx 2a(t-t_0)z_t - C z_t^3\;.
\]
We remark that this assumption about $U^{(4)}(a(1+t))$ should not have much effect. At most, the right-hand side above would have a $+C z_t^3$ term appearing, but in either case this term is very small for $z$ and $t-t_0$ close to zero, which is where most of our analysis will take place.

Making the space and time transformations $q = a(1+t)-q$, $p =  a - p/\ve$ and $t = t-t_0$, as well as normalising constants to one, we arrive at the SDE
\be\label{SDEintro}
\begin{split}
\dd q_t & = p_t\, \dd t\;, \\
\ve^{\beta}\dd p_t & = -p_t\, \dd t + \frac{1}{\ve}(t q_t - q_t^3 + \ve)\, \dd t + \frac{\si}{\sqrt{\ve}}\, \dd W_t \;.
\end{split}
\ee
This will be our main equation for the rest of this paper. By the above discussion, understanding how its solution behaves will be a good indication of the behaviour of the original chain. Equation (\ref{SDEintro}) represents the motion of the particle $q$ in the potential $(1/\ve)V(q,t) \isdefby (1/\ve)(-\frac{1}{2}tq^2 + \frac{1}{4}q^4)$ with an additional $+1$ force giving the particle a small bias towards the right. This force comes from pulling the chain $(q_L,q,q_R)$ and corresponds to the fact that in the absence of noise, $q$ does not just stay at the chain midpoint $a(1+t)$, but lags behind by a small amount.

Rephrasing the discussion after (\ref{eq:SDE before taylor}), the function $V$ represents the energy of a given chain configuration. For negative times, the origin, corresponding to equally spaced particles, minimises $V$. When $t=0$, $V$ undergoes a symmetric pitchfork bifurcation at the origin. For positive times, $V$ has two minima located at $\pm \sqrt{t}$. For $t > 0$ large enough these minima at $\pm \sqrt{t}$ correspond to the configurations where $q$ is a distance $a$ from $q_L$ or $q_R$, respectively, and more than $b$ from the other. In terms of (\ref{SDEintro}), the aim of this paper can be roughly stated as to determine whether $q$ moves towards $+\sqrt{t}$ or $-\sqrt{t}$ as $t$ becomes positive, which corresponds to the chain `breaking', and how it is affected by the speed of pulling.

There are several ways in which one may rigorously define the chain to break. One possible definition is that the chain breaks as soon as the distance between $q$ and one of its neighbours exceeds the range of the pairwise potential $U$. Then the chain either breaks on the right- or left-hand side, depending on whether $q_R - q > b$ or $q - q_L > b$, respectively. This was used in \cite{AB09}. Alternatively, one may consider the chain to break as soon as the chain configuration reaches a neighbourhood of one of the energy minima that emerges after the bifurcation. In the above formalism, this means the process $q_t$ reaching a neighbourhood of $\pm \sqrt{t}$. We shall avoid making this choice by instead giving a precise description of the behaviour of $q_t$ that contains more information than any of these possible definitions.

Equation (\ref{SDEintro}) in full is not something we can treat. But there are two obvious simplifications: the first is to omit the $q_t^3$ term in the equation for $p$, leading to a linear equation that can be solved explicitly; the second is to neglect mass, taking $\ve^{\beta}\dd p = 0$, and consider the overdamped equation. We treat both of these and obtain satisfactory results.

Taking $\si = \ve^{\alpha + 1/2}$ for $\alpha > -1/2$, we firstly consider the linear SDE
\be\label{linearintro}
\begin{split}
\dd q_t^0 & = p_t^0\, \dd t\;, \\
\ve^{\beta}\dd p_t^0 & = -p_t^0\, \dd t + \frac{1}{\ve}(t q_t^0 + \ve)\, \dd t + \ve^{\alpha}\, \dd W_t \;.
\end{split}
\ee
Denoting by $\bbP^{s}$ the law of the solution with vanishing initial condition at time $s < 0$, we have the following result:
\begin{theorem} \label{thm:linear}
Let $q_t^0$ be the solution of (\ref{linearintro}). If $\alpha > 1/4$ then
\[
\lim_{\vez}\liminf_{s \to -\infty}\bbP^{s}\left\{\lim_{t \to \infty}q^0_t = +\infty \right\}  = 1\;,
\]
while if $\alpha < 1/4$ then
\[
\lim_{\vez}\liminf_{s \to -\infty}\bbP^{s}\left\{\lim_{t \to \infty}q^0_t = +\infty \right\}  = \lim_{\vez}\liminf_{s \to -\infty}\bbP^{s}\left\{\lim_{t \to \infty}q^0_t = -\infty \right\} = 1/2\;.
\]
\end{theorem}
This theorem shows that the threshold between fast and slow pulling regimes is given by $\alpha = 1/4$ and is independent of $\beta$. However, we also note that if we start the processes at a finite negative time $-T$, then neglecting mass does have an effect, but only for $\beta < 0$. Indeed, for $-1 < \beta < 0$ and zero initial conditions, the threshold becomes $\alpha = (1+\beta)/4$ (see \cite{Allman}). This result has some clear limitations: for $t>0$, $q_t^0$ shoots off quickly to $\pm \infty$ as the drift becomes ever more repelling, while the solution of (\ref{SDEintro}) is prevented from doing this by the nonlinear term and so is more likely to return to the origin.

Secondly, we neglect the mass term $\ve^{\beta}\dd p$ in (\ref{SDEintro}) and consider the one-dimensional overdamped equation
\be\label{eq:SDEintrooverdamped}
\dd q_t = \frac{1}{\ve}(t q_t - q_t^3 + \ve) \, \dd t + \frac{\sigma}{\sqrt{\ve}} \, \dd W_t \;.
\ee
Again letting $\bbP^{s}$ denotes the law of the solution with vanishing initial condition at time $s < 0$, we have
\begin{theorem}\label{thm:overfull}
Let $q_t$ solve (\ref{eq:SDEintrooverdamped}). There exist constants $c_1,\gamma > 0$ such that if $t_1 = c_1 \sqrt{\ve |\ln \si|}$ then
\begin{enumerate}
	\item (Fast Pulling) for any $\si^{4/3}|\ln \si|^{2/3} \ll \ve(\si) \ll 1$,
	\[
	\lim_{\sez} \liminf_{s \to -\infty}  \bbP^{s}\left\{\inf_{t_1 \leq t}\frac{q_t}{\sqrt{t}} > \gamma \right\} = 1\;.
	\]
	\item (Slow Pulling) for any $\si^2 |\ln \si|^3 \lesssim \ve(\si) \ll \si^{4/3}|\ln \si|^{-13/6}$,
	\[
	  \lim_{\sez}\limsup_{s \to -\infty}\left|\bbP^{s}\left\{\inf_{t_1 \leq t}\frac{q_t}{\sqrt{t}} > \gamma\right\} - {1\over 2}\right| = 0
	\]
	and
	\[
	\lim_{\sez}\limsup_{s \to -\infty}\left|\bbP^{s}\left\{\sup_{t_1 \leq t}\frac{q_t}{\sqrt{t}} < -\gamma\right\} - {1\over 2}\right| = 0\;.
	\]
\end{enumerate}
\end{theorem}
Letting $\si = \ve^{\alpha+1/2}$ above gives $\alpha=1/4$ as the threshold between the different regimes, as found in Theorem \ref{thm:linear}. We also remark that the threshold between fast and slow pulling regimes here differs from that in \cite{AB09}, where it is roughly $\ve = \si$. This difference can be attributed to the bifurcation of $V$: for $t$ near zero, $V$ is almost flat and so the additional $+1$ force requires slower pulling, or stronger noise, to be counteracted than it does in \cite{AB09}, where the potential has positive curvature bounded away from zero.

Having considered two simplifications of (\ref{SDEintro}), we finally return to the full solution itself. Intuitively, by taking $\beta$ large, the effect of the mass term $\ve^{\beta}$ should become small and the solution should behave like that of the overdamped equation (\ref{eq:SDEintrooverdamped}). So if we show that for suitably large $\beta$, the difference between the two solution stays small, we can use Theorem \ref{thm:overfull} to tell us about (\ref{SDEintro}). This leads us to the following result:
\begin{theorem}\label{thm:mass sample paths}
Let $q_t$ solve (\ref{SDEintro}) with $\beta > 2$. There exist constants $c_1, \gamma > 0$, independent of $\beta$ and $\si$, such that for $t_1 = c_1 \sqrt{\ve |\ln \si|}$ and any $t_2 > t_1$,
\begin{enumerate}
	\item (Fast Pulling) if $\si^{4/3}|\ln \si|^{2/3} \ll \ve(\si) \ll 1$ then
	\[
	\lim_{\sez} \liminf_{s \to -\infty} \bbP^{s}\left\{\inf_{t_1 \leq t \leq t_2}\frac{q_t}{\sqrt{t}} > \gamma \right\} = 1\;,
	\]
	\item (Slow Pulling)
	while if $0 < \delta < \beta/2 - 1$ and $\si^{2/(1+2\delta)} \ll \ve(\si) \ll \si^{4/3}|\ln \si|^{-13/6}$ then
	\[
	  \lim_{\sez}\limsup_{s \to -\infty}\left|\bbP^{s}\left\{\inf_{t_1 \leq t \leq t_2}\frac{q_t}{\sqrt{t}} > \gamma\right\} - 1/2 \right| = 0
	\]
	and
	\[
	\lim_{\sez}\limsup_{s \to -\infty}\left|\bbP^{s}\left\{\sup_{t_1 \leq t \leq t_2}\frac{q_t}{\sqrt{t}} < -\gamma\right\} - 1/2 \right| = 0\;.
	\]
\end{enumerate}
\end{theorem}
Note that we only consider finite time intervals here and that there is a slight difference between the lower bound on $\ve$ in (2) above and in Theorem~\ref{thm:overfull}(2). This second point is related to the fact that the mass is of the form $\ve^{\beta}$. However, it does not affect the threshold between the two regimes.

We finally note that (\ref{SDEintro}) is not suitable for considering the chain `breaking' due to a large deviation event. In that case, our expansion of the potential is not valid.

\section{The linear model}\label{sec:linear}
In this section, we give the proof of Theorem \ref{thm:linear} and are therefore dealing with equation 
\eqref{linearintro}. 
This equation is simple enough to have an explicit solution in terms of Airy functions $\Ai(t)$, $\Bi(t)$ (see \cite{dlmf}): using the fact that both $\Ai(z)$ and $\Bi(z)$ solve the equation $w''(z) = z w(z)$, 
and that the Wronskian $\Ai'(t) \Bi(t) - \Bi'(t) \Ai(t) = 1/\pi$, it can be checked that the process
\be\label{eq:qlinear}
\begin{split}
q^0(t) = & \pi \ve^{(1 - 2 \beta)/3}  \left( 
-\Ai (t(\ve,\beta)) 
\int_{-\infty}^t
\e{-\frac{1}{2} (t-s) \ve^{-\beta}} \mathrm{Bi}(s(\ve,\beta))
( \dd s + \eps^\alpha \dd W_s ) \right.\\
& + \left. \mathrm{Bi}(t(\ve,\beta)) 
\int_{-\infty}^t
\e{-\frac{1}{2} (t-s) \ve^{-\beta}} \Ai(s(\ve,\beta))
 ( \dd s + \eps^\alpha \dd W_s ) \right) 
\end{split}
\ee
is the (almost surely unique) solution of \eqref{linearintro} with zero initial conditions and starting time sent to $-\infty$. Above, we have set 
$s(\eps, \beta) = \eps^{-(1+\beta)/3} ( s + \eps^{1-\beta}/4)$. 
The semi-infinite 
stochastic integrals are to be understood as limits of finite stochastic integrals 
from $-T$ to $t$ as $T \to \infty$. The existence of a limit process $q^0(t)$ of this procedure, 
for all $t$ in intervals $[-S,\infty)$, $S>0$, is easy to see in the present case: Clearly, when $-\infty$ is 
replaced by $-T$, \eqref{eq:qlinear} defines a Gaussian Markov process on $[-T,\infty)$, with mean and 
variance readily calculated (albeit with long formulas, which is why we do not display them here). 
By using the asymptotic expansions 
\begin{align}
\Ai(s) & = \frac{1}{2\sqrt{\pi}} s^{-1/4} \e{-2 s^{3/2}/3} (1 + \caO(s^{-2/3}))\;, \qquad s>1 \label{eq:Ai s>0} \\
\Bi(s) & = \frac{1}{\sqrt{\pi}} s^{-1/4} \e{2 s^{3/2}/3} (1 + \caO(s^{-2/3}))\;, \qquad s>1 \label{eq:Bi s>0} \\
\Ai(s) & = \frac{1}{\sqrt{\pi}} |s|^{-1/4} \cos(2 |s|^{3/2}/3 - \pi/4) (1 + \caO(|s|^{-2/3}))\;, \qquad s<-1 
\label{eq:Ai s<0} \\
\Bi(s) & = \frac{-1}{\sqrt{\pi}} |s|^{-1/4} \sin(2 |s|^{3/2}/3 - \pi/4) (1 + \caO(|s|^{-2/3}))\;, \qquad s<-1 
\label{eq:Bi s<0}
\end{align}
it is clear that the distribution of $q_t$ at any time $-S$ converges to a non-degenerate Gaussian as $T \to \infty$,
and thus \eqref{eq:qlinear} is well-defined. 

The behaviour of $q^0$ as $t \to \infty$ can be described in a straightforward way by considering the `renormalized process'
\[
\tilde q(t) = \frac{1}{\pi \eps^{(1-2\beta)/3}} \frac{\e{\frac{1}{2} t \eps^{-\beta}}}{\Bi(t(\eps,\beta))} q^0(t)\;.
\]
Indeed, we then have

\begin{proposition} \label{linear convergence}
$\lim_{t \to \infty} \tilde q_t$ exists almost surely, and is a Gaussian random variable 
with mean 
\be \label{linear mean}
m = \eps^{(1+\beta)/3} \e{-\frac{1}{12} \eps^{1-2\beta}}
\ee
and variance 
\be \label{linear var}
v =  \eps^{2 \alpha + (1+\beta)/3} \e{-\frac{1}{4} \eps^{1-2\beta}} \int_{-\infty}^{\infty}
\e{s \eps^{(1-2\beta)/3}} \Ai(s)^2 \, \dd s\;.
\ee
\end{proposition}

\begin{proof}
Let us first investigate the deterministic integrals in \eqref{eq:qlinear}. Since $\Bi(s) < \Bi(t)$ for 
$s < t$ and $t \geq 0$, we have that the deterministic integral in the first line of \eqref{eq:qlinear}, 
after renormalisation, is bounded 
by $\Ai(t(\eps,\beta)) \int_{-\infty}^t \e{s\ve^{-\beta}/2} \, \dd s$ for large enough $t$, and thus converges to zero as 
$t \to \infty$ due to \eqref{eq:Ai s>0}. It is known \cite{dlmf} that 
$\int_{-\infty}^\infty \e{ps} \! \Ai(s) \, \dd s = \e{p^3/3}$ for all $p > 0$, and thus 
the limit of the corresponding (renormalized) integral in the 
second line of \eqref{eq:qlinear} is given by 
\[
\int_{-\infty}^\infty \! \e{\frac{s}{2} \eps^{-\beta}} \Ai(s(\eps,\beta)) \, \dd s = 
\eps^{(1+\beta)/3} \e{-\frac{1}{8} \eps^{1-2\beta}} \int_{-\infty}^\infty \! \e{\frac{1}{2} s \eps^{(1-2\beta)/3}}
\Ai(s) \, \dd s = \eps^{(1+\beta)/3} \e{-\frac{1}{12} \eps^{1-2\beta}}\;.
\]
Clearly, this is also the limit of $E(\tilde q(t))$ as $t \to \infty$. 

The stochastic integrals in $\tilde q(t)$ are given by 
\[
J_1(t) = \ve^{\alpha}h(t) \int_{-\infty}^t f_1(s) \, \dd W_s\;, \qquad J_2(t) = \ve^{\alpha}\int_{-\infty}^t f_2(s) \, \dd W_s\;,
\]
with 
\[
h(t) = \frac{\Ai(t(\eps,\beta))}{\Bi(t(\eps,\beta))}\;, \quad f_1(s) = \Bi(s(\eps,\beta)) \e{s \eps^{-\beta}/2}\;, 
\quad f_2(s) = \Ai(s(\eps,\beta)) \e{s \eps^{-\beta}/2}\;.
\]
By the time change $\tau_1(t) = \eps^{2 \alpha} \int_{-\infty}^t f_1^2(s) \, \dd s$, 
$J_1(t)$ equals $h(t) B_{\tau_1(t)}$ in distribution, where $B_s$ is a standard Brownian motion. By the law 
of the iterated logarithm, 
\be\label{lil}
\limsup_{t \to \infty} \frac{|J_1(t)|}{h(t) \sqrt{2 \tau_1(t) \ln \ln \tau_1(t)}} = 1
\ee
almost surely. Again using $\Bi(s) < \Bi(t)$ for $s < t$ and $t \geq 0$, we find 
\[
h(t) \sqrt{\tau_1(t)} \leq \Ai(t(\eps,\beta))\ve^{\alpha} \left(\int_{-\infty}^t \ve^{s \eps^{\beta}} \, \dd s\right)^{1/2}\;,
\]
which by \eqref{eq:Ai s>0} converges to zero superexponentially fast. By \eqref{eq:Bi s>0} it is easy to see 
that $\ln \ln \tau_1(t)$ grows only proportionally to $\ln t$, and thus the denominator on the left-hand 
side of \eqref{lil} converges to zero. It follows that $J_1(t) \to 0$ as $t \to \infty$ almost surely. 
$J_2$, on the other hand, is a square-integrable martingale, and thus converges almost surely. Each $J_2(t)$ is 
Gaussian, and thus so is the limit. It has mean zero and variance 
\[
v = \ve^{2\alpha}\int_{-\infty}^{\infty}\! f_2^2(s) \, \dd s =  \eps^{2 \alpha} \int_{-\infty}^{\infty} \! 
\e{s \eps^{-\beta}} \Ai(s(\eps,\beta))^2 \, \dd s\;.
\]
The same change of variable that was employed to get \eqref{linear mean} yields \eqref{linear var}.
\end{proof}

Since $\Bi(t) \e{-t}$ diverges as $t \to \infty$, Proposition \ref{linear convergence} 
means $\lim_{t \to \infty} |q^0(t)| = \infty$ almost surely. Whether the divergence is to plus or
minus infinity is determined by the sign of $\tilde q_{\infty} = \lim_{t \to \infty} \tilde q_t$. 
$\tilde q_{\infty}$ is a Gaussian random variable with mean $m = m(\eps) >0$ given by \eqref{linear mean}, 
and variance $v = v(\eps)$ given by \eqref{linear var}. Thus if 
$\lim_{\eps \to 0} m(\eps) / \sqrt{v(\eps)} = \infty$, then 
$\lim_{\eps \to 0} \bbP(\lim_{t \to \infty} q^0(t) = +\infty ) = 1$, as the distribution of $\tilde q_\infty$ 
concentrates on the positive half line. On the other hand, if $\lim_{\eps \to 0} m(\eps) / \sqrt{v(\eps)} = 0$, then 
$\lim_{\eps \to 0} \bbP(\lim_{t \to \infty} q^0(t) = +\infty ) = 1/2$, as the distribution of $\tilde q_\infty$
becomes spread out and $\bbP(\tilde q_\infty > 0) \to 1/2$ as $\eps \to 0$.

We now determine the circumstances under which each of the above cases occurs. Define 
\[
J(p) = \int_{-\infty}^{\infty} \e{2ps} \Ai^2(s) \, \dd s\;.
\]
We have
\begin{lemma}
There exist constants $c_1$ and $c_2$ such that 
\begin{itemize}
\item[(i)] $\lim_{p \to \infty} p^{1/2} \e{-2 p^3/3} J(p) = c_1$,
\item[(ii)] $\lim_{p \to 0} p^{1/2} \e{-2 p^3/3} J(p) = c_2$.
\end{itemize}
\end{lemma}
\begin{proof}
Consider first the case $p \to \infty$. Then, 
\[
\int_{-\infty}^1 \e{2ps} \Ai^2(s) \, \dd s \, \, p^{1/2} \e{-2 p^3/3} \leq C \e{-2 p^3 / 3 + 2 p} p^{1/2} \to 0
\]
as $p \to \infty$. For $s>1$ we use \eqref{eq:Ai s>0} and find
\[
\begin{split}
\int_1^\infty \e{2ps} \Ai^2(s) \, \dd s & = \frac{1}{4 \pi} \int_1^\infty \e{-4 s^{3/2}/3 + 2ps} s^{-1/2}
(1 + \caO(s^{-3/2})) \, \dd s \\ 
& = \frac{p}{4 \pi} \int_{1/p^2}^\infty \e{-p^3 (4 t^{3/2}/3 - 2 t)} t^{-1/2} (1 + \caO(p^{-3} t^{-3/2})) \, \dd t\;,
\end{split}
\]
where we used the substitution $ s = p^2 t$. Decompose the integral as $\int_{1/p^2}^{\infty} = \int_{1/p^2}^{1/4} + \int_{1/4}^{\infty}$. The first of these is bounded by $C\e{p^3/3}$ for some $C>0$ and we can ignore it. For the second, we have $\caO(p^{-3} t^{-3/2}) = \caO(p^{-3})$ and can take this outside the integral. Then by the Laplace method, 
\[
\begin{split}
\int_{1/4}^\infty \e{-p^3 (4 t^{3/2}/3 - 2 t)} t^{-1/2} \, \dd t & = \e{2 p^3/3} \int_{1/4}^\infty
\e{-p^3 [ (t-1)^2 + \caO(t-1)^3]} t^{-1/2} \, \dd t\\  
&= \e{2 p^3/3} p^{-3/2} \sqrt{\pi} (1 + \caO(1/p))\;.
\end{split}
\]
Thus (i) holds with $c_1 = \frac{1}{4 \sqrt{\pi}}$. For (ii), we use that 
\[
\int_{-1}^\infty \e{2ps} \Ai^2(s) \, \dd s \to \int_{-1}^\infty \Ai^2(s) \, \dd s = \rm{const}
\]
as $p \to 0$. With \eqref{eq:Ai s<0} we then get 
\[
\begin{split}
\int_{-\infty}^{-1} \e{2ps} \Ai^2(s) \, \dd s & = \frac{1}{\pi} \int_{-\infty}^{-1}
\e{2ps}  |s|^{-1/2} \cos^2(2 |s|^{3/2}/3 - \pi/4)  (1 + \caO(|s|^{-3/2})) \, \dd s \\
& = \frac{1}{\pi \sqrt{p}} \int_p^\infty \e{-2t} t^{-1/2} 
\cos^2(2 p^{-3/2} t^{3/2}/3 - \pi/4) \, \dd t  + \caO(1)\;,
\end{split} 
\]
where in the last line we used the substitution $t = -ps$. As $p \to 0$, the integral in the 
last line above converges to 
\[
\frac{1}{2} \int_0^\infty t^{-1/2} \e{-2t} \, \dd t = \frac{\sqrt{\pi}}{2 \sqrt{2}}\;,
\]
which proves (ii) with $c_2 = \frac{1}{\sqrt{\pi} 2^{3/2}}$.
\end{proof}
When substituting $p = \eps^{(1-2 \beta)/3}/2$ into the last lemma, we find that
\[
v(\eps) =  C \eps^{2 \alpha + (1+\beta)/3} \e{-\frac{1}{4} \eps^{1-2\beta}} 
\eps^{(1-2\beta)/6} \e{\frac{1}{12} \eps^{1-2\beta}} = C
\eps^{2 \alpha +  \frac{2\beta}{3} + \frac{1}{6}} \e{-\frac{1}{6} \eps^{1-2\beta}}\;.
\]
Thus, $m(\eps) / \sqrt{v(\eps)} = \tilde C \eps^{-\alpha + 1/4}$, which completes the proof of Theorem 
\ref{thm:linear}.

\section{The Overdamped Model}\label{sec:over}

The proof of Theorem~\ref{thm:overfull} consists of two parts. First, we show that given an arbitrary compact set $\caX$ containing a neighbourhood of the origin
(say $\caX = [-1,1]$), and given an arbitrary negative time $-T$, the process $q_t$ starting at $0$ at time $-\infty$ will be in $\caX$ with
very high probability at time $-T$. In a second part, we then take advantage of the Markovian nature of the process to restart it at time $-T$ and
to show that the conclusion of Theorem~\ref{thm:overfull} holds uniformly over initial conditions belonging to $\caX$ at time $-T$.
These two steps are formulated as Proposition~\ref{prop:remainbounded} and Proposition~\ref{prop:over} below. 
Theorem~\ref{thm:overfull} is then an immediate consequence of these
two results.

The first part can be formulated as follows:

\begin{proposition}\label{prop:remainbounded}
Let $\caX = [-1,1]$, let $T \geq 1$, and let $q_t$ solve (\ref{eq:SDEintrooverdamped}). Then, we have
\begin{equ}
\lim_{\sigma, \eps \to 0} \liminf_{s \to -\infty} \bbP^{s} \bigl(q_{-T} \in \caX\bigr) = 1\;.
\end{equ}
\end{proposition}

\begin{proof}
Applying It\^o's formula to the function $q \mapsto q^2$, we obtain
\begin{equ}
{\dd \over \dd t} \bbE q_t^2 = \bbE \Bigl({2t\over \eps} q_t^2 - {2\over \eps} q_t^4 + 2q_t + {\sigma^2 \over \eps}\Bigr)
\leq -{1\over \eps} \bbE q_t^2 + \eps + {\sigma^2 \over \eps}\;,
\end{equ}
where we made use of the fact that $t \leq -1$. The claim then follows at once from the fact that we assumed 
that $\eps \to 0$ and $\sigma \to 0$.
\end{proof}

The remainder of this section is devoted to the proof of the following statement, where we denote by $\bbP^{-T,x}$ the
law of (\ref{eq:SDEintrooverdamped}) with initial condition $q_{-T} = x \in \caX$.

\begin{proposition}\label{prop:over}
Let $q_t$ solve (\ref{eq:SDEintrooverdamped}). There exist constants $c_1,\gamma > 0$ such that if $t_1 = c_1 \sqrt{\ve |\ln \si|}$ then
\begin{enumerate}
	\item (Fast Pulling) for any $\si^{4/3}|\ln \si|^{2/3} \ll \ve(\si) \ll 1$,
	\[
	\lim_{\sez} \inf_{x \in \caX} \bbP^{-T,x}\left\{\inf_{t_1 \leq t}\frac{q_t}{\sqrt{t}} > \gamma \right\} = 1\;.
	\]
	\item (Slow Pulling) for any $\si^2 |\ln \si|^3 \lesssim \ve(\si) \ll \si^{4/3}|\ln \si|^{-13/6}$,
	\[
	  \lim_{\sez}\sup_{x \in \caX}\left|\bbP^{-T,x}\left\{\inf_{t_1 \leq t}\frac{q_t}{\sqrt{t}} > \gamma\right\} - 1/2\right| = 0
	\]
	and
	\[
	\lim_{\sez}\sup_{x \in \caX}\left|\bbP^{-T,x}\left\{\sup_{t_1 \leq t}\frac{q_t}{\sqrt{t}} < -\gamma\right\} - 1/2\right| = 0\;.
	\]
\end{enumerate}
\end{proposition}

Our approach in this section is based on that developed by Berglund and Gentz in \cite{BG02, BG}. They consider similar equations to (\ref{eq:SDEintrooverdamped}), but with drift terms $(1/\ve)f(q,t)$ such that $f(q,t) = -f(-q,t)$ and $f(0,0) = \partial_q f(0,0) = 0$. A simple example of such an $f$ is $f(q,t) = tq - q^3$. Our additional drift term arising from pulling means that we cannot directly apply their results, except in a few cases as will be made clear. 

\subsection{Fast pulling}\label{sec:fast}
We begin by considering the fast pulling regime from Proposition~\ref{prop:over}. In this case, the noise in the system is not strong enough to overcome the asymmetry caused by pulling and the qualitative behaviour of $q_t$ is the same as that of the deterministic solution $q_t^{\deter}$ of the ODE
\be\label{eq:qtdet}
\dot q_t^{\deter} = \frac{1}{\ve}(t q_t^{\deter} - (q_t^{\deter})^3 + \ve)\;, \quad q_{-T}^{\deter} = x\;.
\ee
In particular, we will see that $q_t^{\deter}$ falls into the right-hand well by a time of order $\sqrt{\ve |\ln \si|}$ after the bifurcation and so too does $q_t$. The strategy is as follows:
\begin{enumerate}
	\item Show that $q_t$ is of order $\sqrt{\ve}$ when $t = \sqrt{\ve}$.
	\item Show that $(q_t,t)$ then leaves the space-time set $\caK(\kappa)$ (see (\ref{eq:Kkappa})), by a time of order $\sqrt{\ve|\ln \si|}$.
	\item Show that $q_t$ approaches the right-hand well and stays in a small neighbourhood of it up until any time $t_2 > 0$.
	\item Show that by taking $t_2$ large enough, $q_t$ stays in a neighbourhood of the right-hand well of order $t^{1/2-\gamma}$ for any $0 < \gamma < 1/2$ and all $t \geq t_2$.
\end{enumerate}

\subsubsection{Step One:}
We begin by describing how $q_t^{\deter}$ behaves.
\begin{lemma}\label{lem:qtdet}
Let $q^{\deter}_t$ be the solution of (\ref{eq:qtdet}). Then we have, uniformly for all initial conditions $x \in \caX$,
\[
q^{\deter}_t \asymp
\begin{cases}
\ve/|t|\text{  for  }-T + \ve|\ln \ve| \leq t \leq -\sqrt{\ve}\\
\sqrt{\ve}\text{  for  }-\sqrt{\ve} \leq t\leq \sqrt{\ve}
\end{cases}
\]
and there is a constant $C>0$ such that for all $x \in \caX$, $|q_t^{\deter}| < C$ for all $-T \leq t \leq -T + \ve|\ln \ve|$.
\end{lemma}

\begin{proof}
Consider the equation
\be\label{eq:x*}
tq_+^*(t) - q_+^*(t)^3 + \ve = 0\;.
\ee
By fixing $t$ and differentiating the left-hand side with respect to $q_+^*$, we see that for $t < 0$ it has no turning points and so admits a unique real-valued solution. Furthermore, we can check that $q_+^*(t) \asymp \ve$ and $(q_+^*)'(t) = q_+^*(t)/(3q_+^*(t)^2 - t) \asymp \ve$ for negative $t$ bounded away from zero.

Suppose first that the initial condition satisfies $x \geq q_+^*(-T)$. Define $z_t = q_t^{\deter} - q_+^*(t)$. As long as $z_t \geq 0$, we have $\dot z_t \leq tz_t/\ve$ so that
\[
0 \leq q_t^{\deter} - q_+^*(t) \leq (x - q_+^*(-T))\e{(t^2-T^2)/2\ve}\;.
\]
Let $t_0 = -T + \ve |\ln \ve|$. If $z_t < 0$ for some $-T < t < t_0$, which means that $q_t^{\deter} \asymp \ve$, then the analysis below for $t \geq t_0$ can be applied from that time. Otherwise, the above inequality shows that $q_{t_0}^{\deter} \asymp \ve$. 

If $x \leq q_+^*(-T)$ then we define $z_t = q_+^*(t) - q_t^{\deter}$. As $(q_+^*)'(t) > 0$ for negative $t$ bounded away from zero, we have $z_t \geq 0$ for such $t$. In this case, there is $c_1 > 0$, independent of $x \in \caX$, such that
\[
\dot z_t \leq c_1 \ve + \frac{1}{\ve}(tz_t + 3q_+^*(t) z_t^2)
\]
for all $-T \leq t \leq t_0$. Furthermore, as long as $z_t \leq -t/6q_+^*(t)$ (which is satisfied by $z(-T)$ for all $x \in \caX$ by taking $\ve$ sufficiently small) then $3q_+^*(t) z_t^2 \leq -tz_t/2$ and so
\[
\dot z_t \leq c_1 \ve + \frac{1}{2\ve}tz_t\;.
\]
This tells us
\[
0 \leq q_+^*(t) - q_t^{\deter} \leq (q^*(-T) - x)\e{(t^2-T^2)/4\ve} + c_1 \ve \int_{-T}^t \! \e{(t^2-s^2)/4\ve}\, \dd s\;,
\]
which shows that $z_t \leq -t/6q_+^*(t)$ for all $-T \leq t \leq t_0$ and we can use the above inequality to again see that $q_{t_0}^{\deter} \asymp \ve$.

We now analyse the behaviour for $t \geq t_0$. As $q_{t_0}^{\deter} > 0$ and $\dot q_t^{\deter} = 1$ whenever $q_t^{\deter} = 0$, it follows that $q_t^{\deter} \geq 0$ for all $t \geq t_0$. Therefore, for $t \geq t_0$ we have
\[
\dot q_t^{\deter} \leq \frac{1}{\ve}(t q_t^{\deter} + \ve)
\]
and so
\be\label{eq:qdetupper}
q_t^{\deter} \leq q_{t_0}^{\deter}\e{(t^2-t_0^2)/2\ve} + \int_{t_0}^t \e{(t^2-s^2)/2\eps}\dd s \leq
 \begin{cases} c_2\ve/|t|\text{  for  }t_0 \leq t \leq -\sqrt{\ve}\\
 c_2\sqrt{\ve}\text{  for  }-\sqrt{\ve} \leq t\leq \sqrt{\ve}
\end{cases} 
\ee
for some constant $c_2 > 0$. To obtain the lower bound, we use that for $t\leq 0$, $tq - q^3 \geq 2tq$ as long as $q^2 \leq |t|$. By taking $\ve$ sufficiently small, we have $0 < q_{t_0} \leq \sqrt{|t_0|}$ for all initial conditions $x \in \caX$. As long as $0 \leq q_t^{\deter} \leq \sqrt{|t|}$, then
\[
\dot q_t^{\deter} \geq \frac{1}{\ve}(2 t q_t^{\deter} + \ve)
\]
and
\[
q_t^{\deter} \geq q_{t_0}\e{(t^2-t_0^2)/\ve} + \int_{t_0}^t \e{(t^2-s^2)/\ve}\dd s\;.
\]
By (\ref{eq:qdetupper}), we certainly have $q_t^{\deter} \leq \sqrt{|t|}$ for $t \leq -\sqrt{\ve}$ and so the above inequality gives the corresponding lower bound for $q_t^{\deter}$ up until this time. For $-\sqrt{\ve} \leq t \leq \sqrt{\ve}$, we then have $t q_t^{\deter} - (q_t^{\deter})^3 \geq -C\ve$ for some constant $C>0$, so that $q_t^{\deter}$ remains of order $\sqrt{\ve}$ in this interval. This completes the proof.

\end{proof}

We now show that the deviation process $y_t \isdefby q_t - q_t^{\deter}$ satisfies $|y(\sqrt{\ve})| < h \ve^{-1/4}$ for some $h \ll \ve^{3/4}$, which will complete Step One. The process $y_t$ solves
\be\label{eq:y}
\dd y_t = \frac{1}{\ve}[a(t)y_t + b(y_t,t)]\, \dd t + \frac{\si}{\sqrt{\ve}}\, \dd W_t\;, \quad y(-T) = 0\;,
\ee
where $a(t) = t-3(q_t^{\deter})^2$ and $b(y_t,t) = -3q_t^{\deter}y_t^2 - y_t^3$. For all pairs $(y,t) \in \caB(h)$ for a choice of $h = \caO(\ve^{1/4})$ (see (\ref{eq:Bh}) and Lemma \ref{lem:atxit}), we have $|b(y,t)| \leq My^2$. Solving (\ref{eq:y}) gives
\be\label{eq:decomposition}
y_t = \frac{\si}{\sqrt{\ve}}\int_{-T}^t \e{\alpha(t,s)/\ve}\, \dd W_s + \frac{1}{\ve}\int_{-T}^t \e{\alpha(t,s)/\ve}b(y_s,s)\, \dd s \bydefis y_t^0 + y_t^1\; ,
\ee
where $\alpha(t,s) = \int_s^t a(u)\, \dd u$.

We now define the space-time set $\caB(h)$ mentioned above. If we write $\Var(y_t^0) = \si^2 v(t)$, then we find that $v(t)$ solves the ODE
\[
\ve \dot v = 2a(t)v + 1\;, \quad v(-T) = 0\; .
\]
Let $\xi(t)$ be a particular solution of this ODE with nonzero initial condition, given by
\be\label{eq:xi}
\xi(t) = \xi(-T)\e{2\alpha(t,-T)/\ve} + \frac{1}{\ve}\int_{-T}^t \e{2\alpha(t,s)/\ve}\, \dd s\;,\qquad \xi(-T) = \frac{1}{2|a(-T)|} \; .
\ee
Then we define
\be\label{eq:Bh}
\caB(h) = \{(y,t): -T \leq t \leq \sqrt{\ve}, |y| < h\sqrt{\xi(t)}\}
\ee
and the stopping time $\tau_{\caB(h)} = \inf\{t \geq -T: (y_t,t) \notin \caB(h)\}$. Before estimating $\tau_{\caB(h)}$, we must first understand how $a(t)$ and $\xi(t)$ behave:
\begin{lemma}\label{lem:atxit}
Let $q_t^{\deter}$ solve (\ref{eq:qtdet}), define $a(t) = t - 3(q_t^{\deter})^2$ and let $\xi(t)$ be given by (\ref{eq:xi}). Then, uniformly for $x \in \caX$, $a(t) \asymp t$ for $-T \leq t \leq -\sqrt{\ve}$ and $|a(t)| = \caO(\sqrt{\ve})$ for $|t| \leq \sqrt{\ve}$. We also have, uniformly for $x \in \caX$,
\[
\xi(t) \asymp \frac{1}{|t|\vee \sqrt{\ve}}
\]
and $|\dot \xi(t)| = \caO(1/\ve)$ for all $-T \leq t \leq \sqrt{\ve}$.
\end{lemma}
\begin{proof}
The assertions about $a(t)$ follow from Lemma \ref{lem:qtdet}. We can use this to tell us how $\xi(t)$ behaves, for which it is helpful to consider the case $a(t)=t$ as an example. Furthermore, since $\xi$ solves $\ve \dot \xi = 2a(t)\xi + 1$, this tells us that $|\dot \xi(t)| = \caO(1/\ve)$.
\end{proof}

Having established the behaviour of all relevant quantities, we can now prove the following proposition telling us that sample paths are likely to remain in $\caB(h)$ for all times $t \leq \sqrt{\ve}$.
\begin{proposition}\label{thm:Bh}
There exists a constant $C>0$ such that for all $\ve$ sufficiently small, all $\si < h \ll \ve^{3/4}$ and all initial conditions $x \in \caX$,
\[
\bbP^{-T,0}\{\tau_{\caB(h)} < \sqrt{\ve}\} \leq \frac{C}{\ve^2}\exp\left\{-\frac{h^2}{2\si^2 }(1 - r(h,\ve))\right\} \; ,
\]
where $r(h,\ve) = \caO(\sqrt{\ve}) + \caO(h\ve^{-3/4})$ uniformly for $x \in \caX$.
\end{proposition}
\begin{remark}
Choosing $h = k \, \si \sqrt{|\ln \si|}$ with $k > 0$ large enough guarantees that the right-hand side tends to zero as $\sez$ and that $h\sqrt{\xi(\sqrt{\ve})} \ll \sqrt{\ve}$, in which case we may take $q(\sqrt{\ve}) \asymp \sqrt{\ve}$ uniformly for all $x \in \caX$.
\end{remark}
\begin{proof}
Recall the decomposition $y_t = y^0_t + y_t^1$ from (\ref{eq:decomposition}). We have for all $t < \tau_{\caB(h)} \wedge \sqrt{\ve}$,
\begin{align*}
\frac{|y^1_t|}{\sqrt{\xi(t)}} & \leq \frac{1}{\sqrt{\xi(t)}}\frac{1}{\ve}\int_{-T}^t \e{\alpha(t,u)/\ve}|b(y_u,u)|\, \dd u\\
& \leq \frac{M h^2}{\sqrt{\xi(t)}}\left(\sup_{-T\leq u\leq t}\xi(u)\right)\frac{1}{\ve}\int_{-T}^t \e{\alpha(t,u)/\ve}\, \dd u \\
& \leq \frac{c_1 M h^2}{\ve^{3/4}}
\end{align*}
for some constant $c_1 > 0$, where we obtain the final inequality by bounding
\[
\frac{1}{\ve}\int_{-T}^t \e{\alpha(t,u)/\ve}\, \dd u \leq
\begin{cases}
C/|t| \text{ for } t \leq -\sqrt{\ve}\\
C/\sqrt{\ve} \text{ for } |t| \leq \sqrt{\ve}
\end{cases}
\]
and
\[
\frac{1}{\sqrt{\xi(t)}}\left(\sup_{-T\leq u\leq t}\xi(u)\right) \leq
\begin{cases}
C/\sqrt{|t|} \text{ for } t \leq -\sqrt{\ve}\\
C\ve^{-1/4} \text{ for } |t| \leq \sqrt{\ve}
\end{cases}
\]
Therefore, if $|y_t^0|/\sqrt{\xi(t)} < h(1 - c_1Mh\ve^{-3/4})$ for all $-T \leq t \leq \sqrt{\ve}$ then we must have $\tau_{\caB(h)} > \sqrt{\ve}$. Letting $H = h(1 - c_1Mh\ve^{-3/4})$, we obtain exactly as in the proof of Proposition 4.3 from \cite{BG02} that for sufficiently small $\ve$,
\[
\bbP^{-T,0}\left\{\sup_{-T \leq t \leq \sqrt{\ve}} \frac{|y^0_s|}{\sqrt{\xi(t)}} > H \right\} \leq \frac{C}{\ve^2}\exp\left\{-\frac{H^2}{2\si^2}(1 - \caO(\sqrt{\ve}))\right\}\; .
\]
for some $C>0$. Note we cannot apply that proposition directly because our function $a(t)$ behaves differently for $|t| \leq \sqrt{\ve}$ than the corresponding function there. In particular, here $a(t) < 0$ for $t \ll \ve$, whereas in \cite{BG02} $a(t) = t + \caO(t^2)$. However, in our case, $|\alpha(t,s)| = \caO((t-s)\sqrt{\ve})$ for $-\sqrt{\ve} \leq s \leq t \leq \sqrt{\ve}$ and we can check that this still allows us to suitably bound the term $P_k$ appearing in equation (4.23) of their proof. 

\end{proof}

\subsubsection{Step Two:}
We define for $\kappa > 0$ the space-time set
\be\label{eq:Kkappa}
\caK(\kappa) = \left\{(q,t): t \geq \sqrt{\ve}, q^2 \geq (1-\kappa)t \right\}\;.
\ee
The boundary of $\caK(\kappa)$ consists of the curves $(\pm \sqrt{(1-\kappa)t},t)$. For the present case, we only need to consider $q \geq 0$ (for the slow pulling regime in Section \ref{sec:slow}, we will also consider $q < 0$). Let $t_0 \geq \sqrt{\ve}$ and suppose that $0 < q(t_0) < \sqrt{(1-\kappa)t_0}$. Then for $t > t_0$ and as long as $0 \leq q_t \leq \sqrt{(1-\kappa)t}$, we have $q_t \geq q_t^{\kappa}$, where
\be\label{eq:xkappa}
\dd q_t^{\kappa} = \frac{1}{\ve}\kappa t q_t^{\kappa}\, \dd t + \frac{\si}{\sqrt{\ve}}\, \dd W_t\;,\quad q^{\kappa}(t_0) = q(t_0)/2 \;.
\ee
Solving this SDE gives
\[
q_t^{\kappa} = \frac{1}{2}q(t_0)\e{(t^2 - t_0^2)/2\ve} + \frac{\si}{\sqrt{\ve}}\, \int_{t_0}^t \! \e{(t^2-s^2)/2\ve}\, \dd W_s \;.
\]
We now state two lemmas that are analogues of Lemmas 3.2.10 and 3.2.11 from \cite{BG} and are proved similarly. For the present section, we will only need to take $t_0 = \sqrt{\ve}$ and, by Step One, $q_0 = q(\sqrt{\ve}) \asymp \sqrt{\ve}$. For the slow pulling regime, other initial conditions will be considered. Let $\tau_{\caK(\kappa)} = \inf\{t \geq t_0: (q_t,t) \notin \caK(\kappa)\}$ and $\tau^0_{\kappa} = \inf\{t \geq t_0 : q_t^{\kappa} \leq 0\}$. 

\begin{lemma}\label{lem:exitK}
Assume that $q_t$ starts at time $t_0 \geq \sqrt{\ve}$ in $q_0 > 0$, where $(q_0,t_0) \in \caK(\kappa)$. Then there is $C>0$, independent of $t_0$, $q_0$ and $\kappa$, such that for all $t \geq t_0 + \ve/t_0$,
\[
\bbP^{t_0,q_0}\{\tau_{\caK(\kappa)} \geq t, \tau_0^{\kappa} \geq t\} \leq \frac{C}{\si }\sqrt{t_0}\sqrt{t}\e{-\kappa (t^2 - t_0^2)/2\ve}\;.
\]
\end{lemma}

\begin{lemma}\label{lem:exitK2}
Let $q_t^{\kappa}$ start at time $t_0 \geq \sqrt{\ve}$ in $q_0/2 > 0$. Then there exist constants $C_1,\,C_2 > 0$, independent of $t_0$, $q_0$ and $\kappa$, such that for all $t \geq t_0$, the probability of reaching zero before time $t$ satisfies the bound
\[
\bbP^{t_0,q_0}\{\tau_0^{\kappa} < t\} \leq \frac{C_1 \si}{q_0 \sqrt{t_0}}\exp\left\{-\frac{C_2 q_0^{\,2} t_0 }{\si^2}\right\}\;.
\]
\end{lemma}

This second lemma shows that when $q_0 \sqrt{t_0} \gg \si$, the linear process $q_t^{\kappa}$ is unlikely to return to zero for any time $t \geq t_0$, while the first lemma shows that as $t$ increases the probability of $q_t$ remaining in $\caK(\kappa)$ decreases. Therefore, we have $q_t \geq q_t^{\kappa}$ until time $\tau_{\caK(\kappa)}$ and so $q_t$ must exit $\caK(\kappa)$ through the curve $\sqrt{(1-\kappa)t}$. Indeed, there are constants $C_1,\,C_2 > 0$ such that for initial time $t_0 = \sqrt{\ve}$ and any initial position $q_0 \asymp \sqrt{\ve}$, we have for any $t \geq 2\sqrt{\ve}$ that
\be\label{eq:exitKupper}
\bbP^{\sqrt{\ve},q_0}\{\tau_{\caK(\kappa)} < t,\, \tau_0^{\kappa} > t\} \geq 1 - \frac{C_1\ve^{1/4}}{\si}\sqrt{t}\e{-\kappa t^2/2\ve} - \frac{C_1 \si}{\ve^{3/4}}\exp \left\{-\frac{C_2 \, \ve^{3/2} }{\si^2}\right\}\;.
\ee
In the present fast pulling regime, $\ve \gg \si^{4/3}$ and so the third term on the right-hand side tends to zero as $\sez$. Picking $t = \sqrt{2k\ve|\ln \si|}$, we see the second term also tends to zero as long as $k > 1/\kappa$. By a time of order $\sqrt{\ve|\ln \si|}$, all paths will have left $\caK(\kappa)$ through its upper boundary.
\subsubsection{Step Three:}
Firstly, we will see how deterministic solutions behave when started from the boundary of $\caK(\kappa)$. For this, we let $q_+^*(t)$ be the same solution of (\ref{eq:x*}) that we considered in the proof of Lemma \ref{lem:qtdet}, i.e. the unique real-valued solution existing for all times $t \geq -T$. For $t>0$ and $\ve$ sufficiently small, we have $\sqrt{t} \leq q_+^*(t) \leq \sqrt{t} + \ve/t$, which can easily be seen by the intermediate value theorem. For the sake of brevity, we shall write $\tau$ to mean $\tau_{\caK(\kappa)}$. The following proposition tells us how $q_t^{\deter,\tau}$ behaves, where $q_t^{\deter,\tau}$ solves
\be\label{eq:xtkappadet}
\dot q_t^{\deter,\tau} = \frac{1}{\ve}(t q_t^{\deter,\tau} - (q_t^{\deter,\tau})^3 + \ve)\;,\qquad q_{\tau}^{\deter,\tau} = \sqrt{(1-\kappa)\tau}\;.
\ee
\begin{proposition}\label{prop:xkappa}
Assume that $\kappa \in (1/2,2/3)$ and let $\eta = 2 - 3\kappa > 0$. There is a constant $C>0$ such that the solution, $q_t^{\deter,\tau},$ of (\ref{eq:xtkappadet}) satisfies
\[
0 \leq q_+^*(t) - q_t^{\deter,\tau} \leq C\left(\frac{\ve}{t^{3/2}} + (q^*(\tau) - q^{\deter,\tau}(\tau))\e{-\eta(t^2 - \tau^2)/2\ve}\right)
\]
for all $t \geq \tau$ and $\ve$ sufficiently small.
\end{proposition}
\begin{remark}
The condition $\kappa > 1/2$ guarantees that paths do not re-enter $\caK(\kappa)$ after leaving, while $\kappa < 2/3$ ensures that the potential is convex outside of $\caK(\kappa)$. 
\end{remark}
\begin{proof}
The inequality $q_t^{\deter,\tau} \leq q_+^*(t)$ follows since $q^{\deter,\tau}(\tau) < q^*_+(\tau)$ and
\be\label{eq:q*+ prime}
(q_+^*)'(t) = \frac{q_+^*(t)}{3q_+^*(t)^2 - t} > 0\;.
\ee
The proof of the other inequality follows along the same lines as that given in \cite[Proposition~4.11]{BG02}. Note, however, that unlike there we only need to take $\ve$ sufficiently small and not $t$. This is because in our case the value of $a_0^*$, which is defined in equation (4.99) of \cite{BG02}, is given by $-2(1+ o_{\ve}(1))$, rather than $-2(1+o_{t}(1))$. Similarly, $M^* = 3(1 + o_{\ve}(1))$. As $q_+^*(t) \leq \sqrt{t} + \ve/t \leq (3/2)\sqrt{t}$ for $t \geq \sqrt{\ve}$ and $\ve$ small, we can use (\ref{eq:q*+ prime}) to show $(q_+^*)'(t) \leq (3/4)t^{-1/2}$, giving $K^* = 3/4$, where $K^*$ is also defined in (4.99). In \cite{BG02}, $K^* = 1/2$, but the proof just requires that $K^* < 1$.
\end{proof}
Now that we understand how $q_t^{\deter,\tau}$ behaves, the final step is to show that $q_t$, starting at the same point, stays close. Having shown that the analogue of Proposition 4.11 from \cite{BG02} holds, the proofs of the subsequent bounds there can easily be extended to our case and we now show what these are. Let
\[
\xi^{\tau}(t) = \frac{1}{2|a^{\tau}(\tau)|}\e{2\alpha^{\tau}(t,\tau)/\ve} + \frac{1}{\ve}\int_{\tau}^t \e{2\alpha^{\tau}(t,s)/\ve}\, \dd s \;,
\]
where $a^{\tau}(\tau) = t - 3(q_t^{\deter,\tau})^2$ is the linearisation of the drift term around $q_t^{\deter,\tau}$ and $\alpha^{\tau}(t,s) = \int_s^t a(u)\, \dd u$. As is shown in Lemma 4.12 from \cite{BG02}, it follows from Proposition \ref{prop:xkappa} above that $|a^{\tau}(\tau)| \asymp t$ so that $\xi^{\tau}(t) \asymp 1/t$.

Now we write
\be\label{eq:Atau}
\caA^{\tau}(h) = \{(q,t): t \geq \tau, |q - q_t^{\deter,\tau}| \leq h \sqrt{\xi^{\tau}(t)}\}
\ee
and let $\tau_{\caA^{\tau}(h)} = \inf\{t \geq \tau: (q_t,t) \notin \caA^{\tau}(h)\}$. The following bound on $\tau_{\caA^{\tau}(h)}$ follows by the analogue of Theorem 2.12 in \cite{BG02}. It tells us that for $\kappa \in (1/2,2/3)$ and any $t_2 > 0$, there exist constants $C,\, h_0 > 0$ such that for $h < h_0 \tau$ and $\ve$ sufficiently small,
\be\label{eq:Atau bound}
\bbP^{\tau,\sqrt{(1-\kappa)\tau}}\{\tau_{\caA^{\tau}(h)}<t_2\} \leq \frac{C}{\ve^2}\exp\left\{-\frac{1}{2}\frac{h^2}{\si^2}\left[1-\caO(\ve) - \caO\left(\frac{h}{\tau}\right)\right]\right\}\;.
\ee
The right-hand side becomes small by choosing $h = k \, \si \sqrt{|\ln \si|}$ for $k$ large enough, for which we note that $\tau \geq \sqrt{\ve}$ by definition so that $h \ll \tau$.

\subsubsection{Step Four:} Let us suppose that $t_2 \geq 1$. By Step Three, we may write $q(t_2) = \sqrt{t_2} + \caO(\si \sqrt{|\ln \si|}) + \caO(\ve)$ independently of $\tau_{\caK(\kappa)}$. For a given $q(t_2)$, let $q_t^{\deter}$ be the corresponding deterministic solution starting at the same point. We can again obtain a similar bound as in Proposition \ref{prop:xkappa}, but for simplicity let us just say that $(9/10)\sqrt{t} \leq q_t^{\deter} \leq (11/10)\sqrt{t}$ for all $t \geq 1$. Letting $y_t = q_t - q_t^{\deter}$, we define $\tau(\gamma) = \inf\{t \geq t_2 : |y_t| > t^{1/2 - \gamma}\}$ for $0 < \gamma < 1/2$. We again decompose $y_t$ into a linear part, $y_t^0$, and nonlinear part, $y_t^1$, as in (\ref{eq:decomposition}). Then $a(t) = t - 3(q_t^{\deter})^2 \asymp -t$ uniformly for all $t \geq 1$ and the function $b(y,t)$ containing the nonlinear terms now satisfies $|b(y_t,t)| < M \sqrt{t}\,y_t^2$ for all $t < \tau(\gamma)$ and some constant $M>0$ independent of $t_2$.  We will show that $\bbP\{\tau(\gamma) < \infty\} \to 0$. For $t \leq \tau(\gamma)$, we have
\begin{align*}
|y_t^1| & \leq \frac{1}{\ve}\int_{t_2}^t \e{\alpha(t,s)/\ve}|b(y_u,u)|\, \dd u\\
& \leq  \frac{M t^{3/2 - 2\gamma}}{\ve}\int_{t_2}^t \e{\alpha(t,s)/\ve}\, \dd u\\
& < C t^{1/2 - 2\gamma}\;,
\end{align*}
where the final inequality holds uniformly in $t$ and the constant $C > 0$ is independent of $t_2$. Therefore, if $|y_t^0| < H(t)$ for all $t \geq t_2$, where $H(t) = t^{1/2-\gamma}(1 - Ct^{-\gamma})$, then we must have $\tau(\gamma) = \infty$. Note that $1 - Ct^{-\gamma} > 0$ and $\dot H (t) > 0$ for all $t \geq t_2$ by taking $t_2$ large enough. We have
\[
\bbP\left\{\sup_{t \geq t_2}\frac{|y_t^0|}{H(t)} > 1 \right\} \leq \sum_{j=0}^{\infty}\bbP \left\{\sup_{s_j \leq t \leq s_{j+1}}\frac{|y_t^0|}{H(t)} > 1 \right\}\;,
\]
where $t_2 = s_0 < s_1 < \ldots$ is chosen by $-\alpha(s_{j+1},s_j) = \ve^2$. Note that, uniformly in $j$, we have $s_{j+1}^2 - s_0^2 \asymp -\alpha(s_{j+1},s_0) = -\alpha(s_{j+1},s_j) - \ldots - \alpha(s_1,s_0) = (j+1)\ve^2$, which shows that $s_j \to \infty$ as $j \to \infty$. Call the summand on the right-hand side above $P_j$. As $H(t)$ is increasing, we can further bound $P_j$ by replacing $H(t)$ with $H(s_j)$. We can also use for $s_j \leq t \leq s_{j+1}$ the inequality
\[
|y_t^0| = \left|\frac{\si}{\sqrt{\ve}}\int_{s_0}^t \! \e{\alpha(t,s)/\ve}\, \dd W_s \right| \leq \e{\alpha(s_j)/\ve}\left|\frac{\si}{\sqrt{\ve}}\int_{s_0}^t \! \e{-\alpha(s)/\ve}\, \dd W_s \right|\;.
\]
This gives for all $j \geq 0$,
\begin{align*}
P_j & \leq \bbP \left\{\sup_{s_0 \leq t \leq s_{j+1}}\left|\int_{s_0}^t \! \e{-\alpha(s)/\ve}\, \dd W_s\right| > \frac{\sqrt{\ve}}{\si}\e{-\alpha(s_j)/\ve}H(s_j) \right\}\\
& \leq 2 \exp\left\{-\frac{\ve \e{2\alpha(s_{j+1},s_j)/\ve}H(s_j)^2}{2\si^2\int_{s_0}^{s_{j+1}}\! \e{2\alpha(s_{j+1},s)/\ve}\,\dd s}\right\}\\
& \leq 2 \exp\left\{-\frac{c_1 s_j H(s_j)^2}{2\si^2 }\right\}\;,
\end{align*}
where the constant $c_1 > 0$ in the final inequality is independent of $j$. Note that the second inequality comes from Lemma B.1.3 in the Appendix of \cite{BG} and the final inequality uses that $\alpha(s_{j+1},s) \asymp -(s_{j+1}^2 - s^2)$ uniformly for all $s_0 \leq s \leq s_{j+1}$ and all $j$. Summing over $j \geq 1$ and using that $s_j - s_{j-1} \geq C\ve^2/s_j$ uniformly in $j$, we have
\begin{align*}
\sum_{j=1}^{\infty}P_j & = \sum_{j=1}^{\infty}P_j \frac{s_j-s_{j-1}}{s_j - s_{j-1}}\\
& \leq \frac{C}{\ve^2}\sum_{j=1}^{\infty} P_j s_j(s_j-s_{j-1})\\
& \leq C\int_{t_2}^{\infty}\! sP(s)\, \dd s\;,
\end{align*}
where
\[
P(s) = 2 \exp\left\{-\frac{c_1 s H(s)^2}{2\si^2 }\right\} \leq 2 \exp\left\{-\frac{c_2 s^{2(1-\gamma)}}{\si^2 }\right\}
\]
and $c_2 > 0$ is a constant. We have used that $sP(s)$ is decreasing when bounding the series by the integral above.  Then
\[
\int_{t_2}^{\infty}\! sP(s)\, \dd s \leq C \si^2 \exp\{-c_2t_2^{\,2(1-\gamma)}/\si^2\}
\]
for some constant $C>0$ depending on $t_2$ and $\gamma$, so that
\[
\bbP\left\{\sup_{t \geq t_2}\frac{|y_t^0|}{H(t)} > 1 \right\} \leq P_0 + \frac{C \si^2}{\ve^2} \exp\{-c_2 t_2^{\,2(1-\gamma)}/\si^2\}
\]
and the right-hand side tends to zero as $\sez$.
\subsection{Slow pulling}\label{sec:slow}
We now consider the slow pulling regime from Proposition~\ref{prop:over}. In this case, the noise dominates the dynamics and cancels out the asymmetry caused by pulling. The process $q_t$ should, therefore, behave similarly to $\tq_t$, where
\be\label{eq:tx}
\dd \tq_t = \frac{1}{\eps}\left(t \tq_t-\tq_t^3\right)\,\dd t + \frac{\si}{\sqrt{\eps}}\,\dd W_t\;,\qquad \tq(-T)= 0\;.
\ee
As we have chosen $\tq(-T)= 0$, the law of $\tq$ is entirely symmetric about zero. The strategy is as follows:
\begin{enumerate}
	\item Recall from \cite{BG02,BG} that $\tq_t$ stays close to the origin with high probability. At time $t = \sqrt{\ve}$, its typical spreading is of order $\si \ve^{-1/4}\sqrt{|\ln \si|}$.
	\item Show that paths of $q_t$ stay close to those of $\tq_t$ until $\tq_t$ leaves the diffusion-dominated strip $\caS(h)$ defined below.
	\item Show that $q_t$ then exits the slightly larger strip, $\caK(\kappa)$, without returning to the origin.
	\item Show that $q_t$ then finally falls into the potential well on the same side as it left $\caK(\kappa)$ and remains there. 
\end{enumerate}
\subsubsection{Step One:}
This step is the same as Step One from the previous section, except now we are analysing the behaviour of $\tq_t$. Unlike in the previous section, we can use directly the results of \cite{BG02,BG}, which we now summarise.  We again define the function $\xi(t)$ as in the last section and now $a(t) \isdefby t - 3(\tq_t^{\deter})^2$, where
\[
\dot \tq^{\deter}_t = \frac{1}{\ve}(t\tq^{\deter}_t - (\tq^{\deter}_t)^3)\;,\qquad \tq^{\deter}_{-T} = 0\;.
\]
Clearly, $\tq^{\deter} \equiv 0$ and so now $a(t) \equiv t$. Again, $\xi(t) \asymp 1/(|t|\wedge \sqrt{\ve})$ for $-T \leq t \leq \sqrt{\ve}$. We define the space-time domain
\[
\caB(h) = \{(\tilde{q},t): -T \leq t \leq \sqrt{\ve}, |\tilde{q}| < h\sqrt{\xi(t)}\}
\]
and the stopping time $\tau_{\caB(h)} = \inf\{t \geq -T: (\tq_t,t) \notin \caB(h)\}$. Applying Theorem 2.10 from \cite{BG02}, we see that there exist constants $C,\, h_0 > 0$ such that for $\ve$ sufficiently small and $h \leq h_0 \sqrt{\ve}$,
\[
\bbP^{-T,0}\{\tau_{\caB(h)} < \sqrt{\ve}\} \leq \frac{C}{\ve^2}\exp\left\{-\frac{h^2}{2\si^2}\left[1-\caO(\sqrt{\ve}) - \caO\left(\frac{h^2}{\ve}\right)\right]\right\}\;.
\]
Choosing $h = k\,\si\sqrt{|\ln \si|}$ for $k$ large enough, the right-hand side tends to zero. At time $\sqrt{\ve}$, we may take $|\tq(\sqrt{\ve})| = \caO(\si \ve^{-1/4}\sqrt{|\ln \si|})$.
\subsubsection{Step Two:}
In the fast pulling section, we saw that at time $\sqrt{\ve}$, $q(\sqrt{\ve}) \asymp \sqrt{\ve}$, from which we could then show its subsequent exit from $\caK(\kappa)$. Now we are considering the exit of $\tq_t$ from $\caK(\kappa)$ and with other values of $\tq(\sqrt{\ve})$ as found in Step One. Before looking at the exit of $\tq_t$ from $\caK(\kappa)$, we must first analyse its exit from a smaller strip. We define for $t \geq \sqrt{\ve}$ the diffusion-dominated strip
\[
\caS(h) = \left\{(\tilde{q},t): t \geq \sqrt{\ve}, |\tilde{q}| < \frac{h}{\sqrt{t}}\right\}
\]
and the stopping time $\tau_{\caS(h)} = \inf\{t \geq \sqrt{\eps}: (\tq_t,t) \notin \caS(h)\}$. See Figure 3 in \cite{BG02} and Figure 3.12 in \cite{BG} for an illustration of $\caS(h)$ and $\caK(\kappa)$ (note that $\caK(\kappa)$ is denoted $\caD(\kappa)$ in \cite{BG02}). Let $h^* \isdefby h_0\,\si \sqrt{|\ln \si|}$, where $h_0 > 0$ is a constant sufficiently large so that $(\tq(\sqrt{\ve}),\sqrt{\ve}) \in \caS(h^*)$. Applying Proposition 4.7 from \cite{BG02} with the choices $h = h^*$ and $\mu = 2$, we see that there exists $C>0$ such that for all $\si$ sufficiently small and all initial conditions $(q_0,\sqrt{\ve}) \in \caS(h^*)$,
\[
\bbP^{\sqrt{\ve},q_0}\{\tau_{\caS(h^*)} \geq t\} \leq \left(\frac{h^*}{\si}\right)^2 \exp\left\{-\frac{2}{3}\frac{(t^2 - \ve)}{2\ve}[1 - \caO(1/\ln(h^*/\si))]\right\}\;,
\]
as long as $\si|\ln \si|^{3/2} = \caO(\sqrt{\ve})$, which we already assume in the slow pulling regime. We can check that by taking $t = \sqrt{2k\eps \ln(h^*/\si)}$ with $k>0$ sufficiently large, the right-hand side tends to zero. For such a choice of $k$, we define $t^* = \sqrt{2k\eps \ln(h^*/\si)}$ and henceforth assume $\tau_{\caS(h^*)} \leq t^*$.

The important point here is that by symmetry, $\tq_t$ exits $\caS(h^*)$ through either boundary with equal probability. Now that we understand the behaviour of $\tq_t$ up until its exit from $\caS(h^*)$, we turn to $q_t$. The following lemma shows that $q_t$ is close to $\tq_t$ at time $\tau_{\caS(h^*)}$.
\begin{lemma}\label{xtx}
If $\ve(\si) \ll \si^{4/3}|\ln \si|^{-13/6}$ and $\tau_{\caS(h^*)} \leq t^*$, then
\[
q_{\tau_{\caS(h^*)}} = \tq_{\tau_{\caS(h^*)}}\left(1+\caO\left(\frac{\ve^{3/4} |\ln \si|^{13/8}}{\si}\right)\right)\;.
\]
\end{lemma}
The proof of this lemma is based on the following simple comparison of $q_t$ and $\tq_t$.
\begin{lemma}\label{lem:comparisons}
Let $q_t$ solve (\ref{eq:SDEintrooverdamped}) with initial condition $q(-T) = x \in \caX$ and let $\tq_t$ solve (\ref{eq:tx}). We have, almost surely, for all $t \geq -T$,
\begin{equation}\label{bounded}
\tq_t + x\e{(t^2-T^2)/2\ve} \leq q_t \leq \tq_t + \int_{-T}^t \e{(t^2-s^2)/2\eps} \, \dd s
\end{equation}
if $x \leq 0$ and
\be\label{bounded2}
\tq_t \leq q_t \leq \tq_t + \int_{-T}^t \e{(t^2-s^2)/2\eps} \, \dd s + x\e{(t^2-T^2)/2\ve}
\ee
if $x \geq 0$.
\end{lemma}
\begin{proof}
We have
\begin{multline}\label{eq:xtcomparison}
q_t = x\e{(t^2-T^2)/2\ve} + \int_{-T}^t \! \e{(t^2-s^2)/2\ve}\, \dd s\\
 - \frac{1}{\ve}\int_{-T}^t \! \e{(t^2-s^2)/2\ve}q_s^{\,3}\, \dd s + \frac{\si}{\sqrt{\ve}}\int_{-T}^t \! \e{(t^2-s^2)/2\ve}\, \dd W_s\;.
\end{multline}
By the comparison principle, $q_t \geq \hat{q}_t$ almost surely, where $\hat{q}_t$ solves
\[
\dd \hat{q}_t = \frac{1}{\ve}(t\hat{q}_t - \hat{q}_t^3)\, \dd t + \frac{\si}{\sqrt{\ve}}\, \dd W_t\;,\qquad \hat{q}(-T) = x\;.
\]
Using this lower bound for $q$ in the second integral on the right-hand side of (\ref{eq:xtcomparison}) gives
\[
q_t \leq \hat{q}_t + \int_{-T}^t \! \e{(t^2-s^2)/2\ve}\, \dd s\;.
\]
When $x \leq 0$, we have $\hat{q}_t \leq \tq_t$ almost surely, which gives the upper bound in (\ref{bounded}). For the lower bound, we have
\begin{align*}
\hat{q}_t & = x\e{(t^2-T^2)/2\ve} - \frac{1}{\ve}\int_{-T}^t \! \e{(t^2-s^2)/2\ve}\hat{q}_s^{\,3}\, \dd s + \frac{\si}{\sqrt{\ve}}\int_{-T}^t \! \e{(t^2-s^2)/2\ve}\, \dd W_s\\
& \geq x\e{(t^2-T^2)/2\ve} - \frac{1}{\ve}\int_{-T}^t \! \e{(t^2-s^2)/2\ve}\tq_s^{\,3}\, \dd s + \frac{\si}{\sqrt{\ve}}\int_{-T}^t \! \e{(t^2-s^2)/2\ve}\, \dd W_s\\
& = x\e{(t^2-T^2)/2\ve} + \tq_t\;.
\end{align*}
The case $x \geq 0$ is easier and does not involve $\hat{q}_t$. It follows along similar lines.
\end{proof}

\begin{proof}[Proof of Lemma \ref{xtx}]
For all $\sqrt{\ve} \leq t \leq t^*$, we have
\be\label{eq:integralerror}
\int_{-T}^t \e{(t^2-s^2)/2\ve}\, \dd s \leq C_1\sqrt{\ve}\left(\frac{h^*}{\si}\right)^4 \leq C_2\sqrt{\ve}\,|\ln \si|^2
\ee
and for all $x \in \caX$,
\[
|x|\e{(t^2-T^2)/2\ve} \leq C_1\left(\frac{h^*}{\si}\right)^4 \e{-T^2/2\ve} \leq C_2\,|\ln \si|^2 \e{-T^2/2\ve}\;.
\]
Of these two estimates, (\ref{eq:integralerror}) gives the larger upper bound. Next observe that
\[
\tau_{\caS(h^*)} \leq t^* \leq C\ve^{1/2}|\ln \si|^{1/4}\;.
\]
Therefore,
\[
\frac{1}{|\tq_{\tau_{\caS(h^*)}}|} = \frac{\sqrt{\tau_{\caS(h^*)}}}{h^*} \leq C\frac{\ve^{1/4}}{\si |\ln \si|^{3/8}}\;.
\]
Using (\ref{eq:integralerror}) and the above inequality together with Lemma \ref{lem:comparisons} gives the result.
\end{proof}
\subsubsection{Step Three:}
Now we analyse the behaviour of $q_t$, rather than $\tq_t$, for $t > \tau_{\caS(h^*)}$. As we saw in the fast pulling case, if $q_t$ starts at time $t_0 \geq \sqrt{\ve}$ at $q_0 > 0$ with $(q_0,t_0)\in \caK(\kappa)$ then $q_t \geq q_t^{\kappa}$ as long as $0 < q_t < \sqrt{(1-\kappa)t}$, where $q_t^{\kappa}$ was defined in (\ref{eq:xkappa}). Unlike in the previous section, we now have to also consider negative initial conditions. We would like a bound of the form $q_t \leq q_t^{\kappa}$ in such cases, but now the bias of $q_t$ in the positive direction makes this more difficult. In order to obtain a corresponding comparison with $q_t^{\kappa}$, we need an additional assumption on $t_0$ and $q_0$ as set out in the following lemma.
\begin{lemma}\label{lem:q0neg}
Suppose that at time $t_0 \geq \sqrt{\ve}$, $q_t$ starts at $q_0 < 0$, where $(q_0,t_0) \in \caK(\kappa)$ and $|q_0| \gg \eps/t_0$. Then we have $q_t \leq q_t^{\kappa}$ as long as $-\sqrt{(1-\kappa)t} \leq q_t \leq 0$ and $\ve$ is sufficiently small.
\end{lemma}
\begin{proof}
For $-\sqrt{(1-\kappa)t_0} < q_0 < 0$, it is certainly true by comparison of the drift and initial conditions that $q_t$ is bounded above by solutions of
\[
\dd z^{\kappa}_t =  \frac{1}{\eps}( \kappa t z^{\kappa}_t + \eps)\, \dd t + \frac{\si}{\sqrt{\eps}}\,\dd W_t\;,\qquad z^{\kappa}_{t_0} = q_0\;,
\]
as long as $-\sqrt{(1-\kappa)t} \leq q_t \leq 0$. The result follows since
\begin{align*}
z^{\kappa}_t & = q_0\e{\kappa(t^2 - t_0^2)/2\eps} + \int_{t_0}^t \e{\kappa(t^2 - s^2)/2\eps}\dd s + \frac{\si}{\sqrt{\eps}}\int_{t_0}^t \e{\kappa(t^2 - s^2)/2\eps}\dd W_s\\
& \leq (q_0 + C\eps/t_0)\e{\kappa(t^2 - t_0^2)/2\eps} + \frac{\si}{\sqrt{\eps}}\int_{t_0}^t \e{\kappa(t^2 - s^2)/2\eps}\dd W_s\\
& \leq \frac{q_0}{2}\e{\kappa(t^2 - t_0^2)/2\eps} + \frac{\si}{\sqrt{\eps}}\int_{t_0}^t \e{\kappa(t^2 - s^2)/2\eps}\dd W_s\\
& = q_t^{\kappa}\;.
\end{align*}
\end{proof}
This lemma shows that, under suitable conditions, we may compare $q_t$ and $q_t^{\kappa}$ for both postive and negative initial conditions $q_0$. For $q_0 < 0$ we get analogous bounds to those in Lemmas \ref{lem:exitK} and \ref{lem:exitK2}. In the previous section, we applied those lemmas with $t_0 = \sqrt{\ve}$ and $q_0 \asymp \sqrt{\ve}$. We now apply these lemmas with $t_0 = \tau_{\caS(h^*)}$ and $|q_0| \asymp h^*/\sqrt{\tau_{\caS(h^*)}}$ (see Lemma \ref{xtx}). Note that if $\sqrt{\eps} \leq \tau_{\caS(h^*)} \leq t^*$ then $\ve/t_0 \ll |q_0|$ and so the conditions of Lemma \ref{lem:q0neg} are satisfied. We obtain a bound similar to (\ref{eq:exitKupper}) and again see that $\caK(\kappa)$ is left by a time of order $\sqrt{\ve|\ln \si|}$.

\subsubsection{Step Four:}
When $q_t$ exits $\caK(\kappa)$ on the positive side, this part is exactly the same as Step Three from the previous section. The other case when $q_t$ exits $\caK(\kappa)$ on the negative side is similar. Firstly, we introduce $q^*_-(t)$, another real-valued solution of (\ref{eq:x*}) existing for $t \geq \sqrt{\ve}$ and satisfying the bounds $-\sqrt{t} \leq q^*_-(t) \leq -\sqrt{t} + \ve/t$ and $(q_-^*)'(t) < 0$ for all such $t$ and $\ve$ small. Note there is also a third real-valued solution of (\ref{eq:x*}) between $q^*_-$ and $q^*_+$, which is an unstable equilibrium branch. By taking $\ve$ small, we have $q_-^*(t) < -\sqrt{(1-\kappa)t}$ for all $t \geq \sqrt{\ve}$. Again writing $\tau = \tau_{\caK(\kappa)}$, we need to check that the deterministic solution, $q_t^{\deter,\tau}$, of (\ref{eq:xtkappadet}), with initial condition $q^{\deter,\tau}(\tau) = -\sqrt{(1-\kappa)\tau}$, satisfies a bound as in Proposition \ref{prop:xkappa} of the form
\[
0 \leq q_t^{\deter,\tau} - q_-^*(t) \leq C\left(\frac{\ve}{t^{3/2}} + (q^{\deter,\tau}(\tau) - q^*(\tau))\e{-\eta(t^2 - \tau^2)/2\ve}\right)
\]
for all $t \geq \tau$. The main thing to ensure is that $q_t^{\deter,\tau}$, which has a bias in the positive direction, does not re-enter the set $\caK(\kappa)$ after having left. To see that this is indeed the case, first note that the derivative with respect to $t$ of the boundary curve, $-\sqrt{(1-\kappa)t}$, is given by $-\frac{1}{2}t^{-1/2}\sqrt{1-\kappa}$. The derivative of $q_t^{\deter,\tau}$ when on the boundary of $\caK(\kappa)$ is given by $\frac{1}{\ve}t^{3/2}(-\kappa+\ve)\sqrt{1-\kappa}$. Using that $t \geq \sqrt{\ve}$, we see that the inequality
\[
\frac{1}{\ve}t^{3/2}(-\kappa+\ve)\sqrt{1-\kappa} < -\frac{1}{2}t^{-1/2}\sqrt{1-\kappa}
\]
holds when $\kappa > 1/2 + \ve$, which is true for all $\kappa > 1/2$ by taking $\ve$ sufficiently small. Having established that $q_-^*(t) \leq q_t^{\deter,\tau} \leq -\sqrt{(1-\kappa)t}$ for all $t \geq \tau$ and $\ve$ sufficiently small, the rest of the proof follows like Proposition 4.11 from \cite{BG02}. The subsequent estimate (\ref{eq:Atau bound}) above showing the concentration of $q_t$ in the set $\caA^{\tau}(h)$ then follows, where $\caA^{\tau}(h)$ was defined in (\ref{eq:Atau}). Finally, we can show as in Step Four from the fast pulling section that $q_t$ stays in a neighbourhood of $-\sqrt{t}$ for all $t \geq t_2$.

\section{The Full Solution}\label{subsec:sample}
We now consider the full equation (\ref{SDEintro}) and show that for sufficiently small mass (large $\beta$), it behaves like the overdamped solution of the previous section. The general strategy is as in Section~\ref{sec:over}, namely we first show that if we start the system
at the origin at time $s \ll -1$, then 
the solution at time $-1$ belongs to a suitable set. This is done in the following two propositions. We then provide a result that 
is uniform over all  solutions starting from the set in question.
The first step is achieved by the following statement:

\begin{proposition}\label{prop:remainboundedfull}
Let $\hat{\caX} = [-\ve^{1-\beta},\ve^{1-\beta}]$, $\hat{\caV} = [-\eps^{-{1 + 3\beta \over 2}},\eps^{-{1 + 3\beta \over 2}}]$, let $T \geq 1$ be a constant, and let $q_t$ solve (\ref{SDEintro}) with $\beta > 2$. Then, we have
\begin{equ}
\lim_{\sigma, \eps \to 0} \liminf_{s \to -\infty} \bbP^{s} \bigl(q_{-2T} \in \hat{\caX}\;,\; p_{-2T} \in \hat{\caV}\bigr) = 1\;.
\end{equ}
\end{proposition}

\begin{proof}
We fix some arbitrary starting time $s < -2T$ and we consider the solution to (\ref{SDEintro}) with initial condition $q_s = p_s = 0$. We define the function $\Psi(p,q,t)$ by
\begin{equ}
\Psi(p,q,t) = {\eps^\beta \over 2}p^2 - {t \over 2\eps} q^2 + {1 \over 4\eps}q^4 - q + {1 \over 2} pq\;.
\end{equ}
Applying It\^o's formula to $\Psi(p_t, q_t, t)$, we obtain
\[
\dd \Psi(p_t,q_t,t) \leq \left(-\ve^{-\beta}\Psi(p_t, q_t, t) + \frac{1}{2 }\si^2\ve^{-1-\beta} - \frac{1}{2\ve}q_t^2 - \frac{1}{2\ve^{\beta}}q_t \right)\, \dd t + \dd M(t)\; ,
\]
where $M$ is some continuous martingale. Using $\ve^{\beta-1}q^2+q \geq -(1/4)\ve^{1-\beta}$, we see that
\begin{equ}
{\dd \over \dd t} \bbE \Psi(p_t,q_t,t) \leq - \eps^{-\beta} \bbE \Psi(p_t,q_t,t) + \frac{1}{2}\si^2 \ve^{-1-\beta} + \frac{1}{8}\ve^{1-2\beta} \;.
\end{equ}
It follows immediately that $\bbE \Psi(p_t,q_t,t) \leq \sigma^2/(2\ve) + (1/8)\ve^{1-\beta}$. Since for $t \leq -2$ we have
\begin{equ}
\Psi(p,q,t) \geq {\eps^\beta \over 4}p^2 + {1 \over 4\eps} q^2 + {1 \over 4\eps}q^4 - {1 \over 4\ve^{\beta}}q^2 - \eps
\geq {\eps^\beta \over 4}p^2 + {1 \over 8\eps} q^2 - \eps - \frac{1}{8}\ve^{1-2\beta}\;,
\end{equ}
it follows that
\begin{align*}
\bbE q_t^2 & \leq 8 \eps^2 + \ve^{2-2\beta} + 4 \sigma^2 + \ve^{2-\beta}\;,\\
\bbE p_t^2 & \leq 4 \eps^{1-\beta} + (1/2) \eps^{1 - 3\beta} + 2\si^2 \ve^{-1-\beta} + (1/2)\ve^{1-2\beta}\;.
\end{align*}
The stated result then follows at once from Chebychev's inequality.
\end{proof}

Now we use the previous proposition to restart the process at time $-2T$. We denote by $\bbP^{\hat{x},\hat{v}}$ the law of the solution of (\ref{SDEintro}) starting at time $-2T$ with $q_{-2T} = \hat{x}$, $p_{-2T} = \hat{v}$.
\begin{proposition}\label{prop:remainboundedfull2}
Let $\caX = [-1,1]$, $\caV = [-\eps^{-\beta},\eps^{-\beta}]$, let $T \geq 1$ be a constant, and let $q_t$ solve (\ref{SDEintro}) with $\beta > 2$. Then, we have
\begin{equ}
\lim_{\sigma, \eps \to 0} \inf_{\hat{x} \in \hat{\caX},\hat{v} \in \hat{\caV}} \bbP^{\hat{x},\hat{v}} \bigl(q_{-T} \in \caX\;,\; p_{-T} \in \caV\bigr) = 1\;.
\end{equ}
\end{proposition}

\begin{proof}
Let
\[
\Psi(p,q,t) = \frac{1}{4}q^2 + \frac{\ve^{\beta}}{2}pq + \frac{\ve^{2\beta}}{2}p^2 + \frac{\ve^{\beta-1}}{4}q^4\;.
\]
Applying It\^o's formula to $\Psi(p_t,q_t,t)$, we obtain
\[
\dd \Psi(p_t,q_t,t) = \frac{1}{\ve}\left(-\frac{\ve^{\beta+1}}{2}p_t^2 + \frac{t}{2}q_t^2 - \frac{1}{4}q_t^4 + \frac{\ve}{2}q_t + \ve^{\beta}t p_t q_t + \ve^{\beta+1}p_t + \frac{\si^2}{2} \right)\, \dd t + \dd M(t)\; , 
\]
where $M$ is some continuous martingale. Using that $t \in [-2T,-T]$ and $\beta > 2$, we have for sufficiently small $\ve$,
\[
\dd \Psi(p_t,q_t,t) \leq \frac{1}{\ve}\left( -\frac{1}{2}\Psi(p_t,q_t,t) + 4\ve^{\beta+1} + \frac{\ve}{2} + \frac{\si^2}{2} \right)\, \dd t + \dd M(t)\; .
\]
It follows that $\bbE \Psi(p_t,q_t,t) \leq \e{-(t+2T)/2\ve}\bbE \Psi(p_{-2T},q_{-2T},-2T) + \si^2 + 8\ve^{\beta+1} + \ve$. We can then use the bounds $\Psi(p,q,t) \geq \ve^{2\beta}p^2/4$ and $\Psi(p,q,t) \geq q^2/8$ to obtain the result.

\end{proof}

We would like to use a singular perturbation approach to show that for $t \geq -T$ and suitably large $\beta$, sample paths of $q_t$ can be approximated by those of an overdamped equation starting at $-T$. For this, we need to first consider a process that is similar to $q_t$ but has better regularity. To this end, we introduce the processes $Q$ and $P$ that solve
\be\label{eq:QP}
\begin{split}
\dd Q_t & = P_t \, \dd t\;,\qquad Q_{-T} = P_{-T} = 0\;,\\
\ve^{\beta}\dd P_t & = -P_t \,\dd t + \frac{\si}{\sqrt{\ve}}\,\dd W_t\;.
\end{split}
\ee
We find that
\[
P_t = \si \ve^{-1/2-\beta}\int_{-T}^t \! \e{-(t-s)\ve^{-\beta}}\, \dd W_s
\]
and
\[
Q_t = \frac{\si}{\sqrt{\ve}}W_t - \ve^{\beta}P_t = \si\ve^{-1/2 - \beta}\int_{-T}^t \! \e{-(t-s)\ve^{-\beta}}W(s)\, \dd s\;.
\]
For $\delta > 0$ we now define two events, $E_1$ and $E_2$, by
\begin{align*}
E_1 & = \{|Q_t| > \si \ve^{-1/2 - \delta} \text{ for some }t \in [-T,t_2]\}\;,\\
E_2 & = \{|P_t| > \si \ve^{-1/2 - \beta/2 - \delta} \text{ for some }t \in [-T,t_2]\}\;.
\end{align*}
\begin{lemma}\label{qp}
There exists $C>0$ such that for all $\delta>0$ and all $t_2 > 0$,
\[
\bbP(E_1 \cup E_2 ) \leq C t_2(\e{-\ve^{-\delta/(t_2+T)}} + \e{-\ve^{-2\delta}/2})
\]
holds for $\si ,\,\ve> 0$ sufficiently small.
\end{lemma}
\begin{proof}
Since $Q_t = \si \ve^{-1/2-\beta}\int_{-T}^t \! \e{-(t-s)\ve^{-\beta}}W(s)\, \dd s$, it follows that
\begin{align*}
\bbP(E_1) & \leq \bbP\{|W_t| > \ve^{-\delta} \text{ for some }t \in [-T,t_2]\}\\
& \leq C t_2 \e{-\ve^{-\delta/(t_2+T)}}.
\end{align*}
where $C>0$ is independent of $\delta$ and $t_2$. We also find that there exists $C>0$ independent of $\delta$ and $t_2$ such that for $\ve$ sufficiently small,
\[
\bbP(E_2) \leq Ct_2\ve^{-2\delta-1}\exp\left\{-\frac{1}{2}\ve^{-2\delta}\right\}.
\]
The result follows by combining these two estimates.
\end{proof}
Now let $y_t = q_t - Q_t$. It solves, almost surely, the second-order ODE
\[
\ve^{\beta}\ddot y = -\dot y + \frac{1}{\ve}(t(y_t+Q_t)-(y_t+Q_t)^3 +\ve)\;,\quad y(-T) = x\;,\, \dot y(-T) = v\;.
\]
The following proposition shows that for almost all paths in $(E_1 \cup E_2)^c$, $y$ may be approximated by the solution of a first-order ODE. Note that the condition on $\beta$ is a little stronger than necessary, but is required later on in this section.
\begin{proposition}\label{y}
For all $t_2 > 1$, there exists $C = C(t_2)>0$ such that for all $\beta > 2$, all $0 < \delta < \beta/2 - 1$, all $\si^{2/(1 + 2\delta)} \ll \ve \ll 1$ and almost all paths in $(E_1 \cup E_2)^c$,
\[
\left|\dot y - \frac{1}{\ve}(t(y_t+Q_t)-(y_t+Q_t)^3 +\ve) \right| \leq C\max\{\ve^{\beta-2 - \delta},\si \ve^{-3/2 + \beta/2 - 2\delta}\}
\]
for all $t \in [-T + 2\ve^{\beta-\delta},t_2]$ and $\si$ sufficiently small.
\end{proposition}
\begin{proof}
If we write $z = \dot y$, then almost surely the pair $(y,z)$, which are differentiable, solve
\begin{align*}
\dot y & = z \;, \\
\ve^{\beta}\dot z & = -z  + \frac{1}{\ve}g(t,y_t+Q_t)\;,
\end{align*}
where $g(t,y_t+Q_t) = t(y_t+Q_t)-(y_t+Q_t)^3 +\ve$. For $t>0$, we have $g(t,\sqrt{t}) \approx 0$ so that we do not expect $y_t + Q_t$, or indeed $y_t$, to be much larger than $\sqrt{t}$. Therefore, we let $\tau = \inf\{t \geq -T: |y_t| >2\sqrt{t_2}\}$. On $(E_1 \cup E_2)^c$, there is $C>0$ depending on $t_2$ such that for all $-T \leq t \leq \tau \wedge t_2$, $|g(t,y_t+Q_t)| < C$. We solve the equation for $z$ to give
\be\label{eq:zt}
z_t = v\e{-(t+T)\ve^{-\beta}} + \ve^{-(1+\beta)}\int_{-T}^t \! \e{-(t-s)\ve^{-\beta}}g(s,y_s+Q_s)\,\dd s
\ee
almost surely, from which we deduce that $|z_t| \leq \ve^{-\beta}\e{-(t+T)\ve^{-\beta}} + C/\ve$ for all $t \leq \tau \wedge t_2$.  This immediately shows that for $-T \leq t \leq (-T + 2\ve^{\beta-\delta}) \wedge \tau$ and sufficiently small $\ve$,
\be\label{eq:yinitial}
|y_t - x| \leq 1 + C\ve^{\beta - \delta - 1}
\ee
and so $\tau > -T + 2\ve^{\beta-\delta}$.

For $-T + \ve^{\beta-\delta} \leq t \leq \tau \wedge t_2$, we have $|z_t| < C/\ve$. Furthermore, for such $t$ we find
\[
 \left|\frac{\dd }{\dd t}g(t,y_t+Q_t)\right| \leq C \max\{\ve^{-1},\si \ve^{-1/2 - \beta/2 - \delta}\}\;.
\]
Now we apply the Laplace method to the integral in (\ref{eq:zt}). For $t \geq -T + 2\ve^{\beta-\delta}$, decompose the integral as
\[
\int_{-T}^t =\int_{-T}^{t-\ve^{\beta-\delta}} + \int_{t-\ve^{\beta-\delta}}^t.
\]
Then, by the boundedness of $g$, we have
\[
\left|\int_{-T}^{t-\ve^{\beta-\delta}} \! \e{-(t-s)\ve^{-\beta}}g(s,y_s+Q_s)\,\dd s\right| < C \e{-\ve^{-\delta}}.
\]
For the remaining integral, we use a Taylor expansion of $g$ to give
\[
g(s,y_s+Q_s) \leq g(t,y_t+Q_t) + C(t-s)\max\{\ve^{-1},\si \ve^{-1/2 - \beta/2 - \delta}\}\;.
\]
Then
\begin{multline*}
\int_{t-\ve^{\beta-\delta}}^t \! \e{-(t-s)\ve^{-\beta}}g(s,y_s+Q_s)\,\dd s \leq \ve^{\beta}g(t,y_t+Q_t)+ \\+ C\ve^{\beta}\e{-\ve^{-\delta}} + C\max\{\ve^{2\beta-1 - \delta},\si \ve^{-1/2 + 3\beta/2 - 2\delta}\}\;,
\end{multline*}
which tells us that for $-T + 2\ve^{\beta-\delta} \leq t \leq \tau \wedge t_2$ and $\si$ sufficiently small,
\be\label{eq:z UB}
z_t \leq \frac{1}{\ve}g(t, y_t+Q_t) + C\max\{\ve^{\beta-2 - \delta},\si \ve^{-3/2 + \beta/2 - 2\delta}\}\;.
\ee
In a similar way, we can also show that
\be\label{eq:z LB}
z_t \geq \frac{1}{\ve}g(t, y_t+Q_t) - C\max\{\ve^{\beta-2 - \delta},\si \ve^{-3/2 + \beta/2 - 2\delta}\}\;.
\ee
We will now show that the assumption $\tau \leq t_2$ leads to a contradiction. For this, we will show that if $y_{\tau} = +2\sqrt{t_2}$, then the right-hand side of (\ref{eq:z UB}) is strictly negative, whereas we should have $z_{\tau} \geq 0$ by continuity. The case $y_{\tau} = -2\sqrt{t_2}$ is similar. First, we note that if $y_{\tau} = +2\sqrt{t_2}$, then there are constants $C_1,\, C_2 >0$ depending on $t_2$ such that
\[
g(\tau, y_{\tau} + Q_{\tau}) \leq -C_1 + C_2\si \ve^{-1/2 - \delta}
\]
and for $\si$ sufficiently small, the right-hand side is strictly negative and bounded away from zero. Then the conditions on $\ve$, $\beta$ and $\delta$ guarantee that the right-hand side of (\ref{eq:z UB}) is strictly negative. This means that we must have $\tau > t_2$ and so (\ref{eq:z UB}) and (\ref{eq:z LB}) hold for all $-T + 2\ve^{\beta-\delta} \leq t \leq t_2$, from which the result follows.
\end{proof}
Now we will use Proposition \ref{y} to tell us something about the SDE (\ref{SDEintro}).
\begin{proposition}\label{th}
For all $t_2 > 1$, there exists $C = C(t_2)>0$ such that for all $\beta > 2$, all $0 < \delta < \beta/2 - 1$, all $\si^{2/(1 + 2\delta)} \ll \ve \ll 1$ and almost all paths in $(E_1 \cup E_2)^c$,
\[
q^-_t \leq q_t + \ve^{\beta}P_t \leq q^+_t
\]
for all $t \in [-T + 2\ve^{\beta - \delta},t_2]$ and $\si$ sufficiently small, where
\be\label{eq:qpm}
\dd q^{\pm}_t = \frac{1}{\ve}(tq^{\pm}_t - (q^{\pm}_t)^3 + \ve(1\pm r(\si)))\, \dd t + \frac{\si}{\sqrt{\ve}}\, \dd W_t\;,\quad q^{\pm}(-T+2\ve^{\beta-\delta}) = x \pm 3\;,
\ee
with $W_t$ the same Brownian motion appearing in (\ref{SDEintro}) and
\[
 r(\si) = C\max\{ \ve^{\beta-2 - \delta},\si \ve^{-3/2 + \beta/2 - 2\delta}\}\;.
\]
\end{proposition}

\begin{proof}[Proof of Theorem \ref{th}]
We will show the upper bound. The lower bound is similar. Letting $t_0 = -T + 2\ve^{\beta - \delta}$, we know by Proposition \ref{y} that there exists $C>0$ and a process $y_t^+$ solving
\[
\dot y_t^+ = \frac{1}{\ve}g(t,y_t^++Q_t) + C\max\{ \ve^{\beta-2 - \delta},\si \ve^{-3/2 + \beta/2 - 2\delta}\}\;,\quad  y^+_{t_0} = x + 2\;,
\]
such that $y_t \leq y_t^+$ for all $t \in [t_0,t_2]$, where $y_t = q_t - Q_t$. Note the initial condition for $y_t^+$ is chosen using (\ref{eq:yinitial}). Therefore, $q_t \leq y_t^+ + Q_t$. In the same way as in Proposition \ref{y}, we can show $|y_t^+| < 4\sqrt{t_2}$ for all $t_0 \leq t \leq t_2$ by considering the sign of $\dot y_t^+$.

Now define $\eta_t \isdefby y_t^+ + Q_t + \ve^{\beta}P_t$ and note that for all $t_0 \leq t \leq t_2$, $|\eta_t| < C$ for some constant $C>0$ depending on $t_2$. It solves
\begin{equation}\label{eta}
\dd \eta_t = \frac{1}{\ve}(t (\eta_t - \ve^{\beta}P_t) - (\eta_t - \ve^{\beta}P_t)^3 + \ve + C\max\{ \ve^{\beta-1 - \delta},\si \ve^{-1/2 + \beta/2 - 2\delta}\})\, \dd t + \frac{\si}{\sqrt{\ve}}\, \dd W_t\;.
\end{equation}
with initial position $\eta_{t_0} \leq x+3$.

Denote the drift term above by $f(t,\eta_t,P_t,\ve)$. We will now show that $f$ is bounded in such a way that allows us to use a comparison principle. As we are working on $(E_1 \cup E_2)^c$, we know that $|\ve^{\beta}P_t| \leq \si \ve^{-1/2 +\beta/2 - \delta}$ for all $t \in [t_0,t_2]$ by definition. Note also that by the conditions on $\delta$ and $\ve$, $\max\{ \ve^{\beta-1 - \delta},\si \ve^{-1/2 + \beta/2 - 2\delta}\} \ll \ve$. We then have
\[
f(t,\eta_t,P_t,\ve) \leq \frac{1}{\ve}(t\eta_t - \eta_t^3 + \ve(1+r(\si)))\;,
\]
where $r(\si) = C\max\{ \ve^{\beta-2 - \delta},\si \ve^{-3/2 + \beta/2 - 2\delta}\}$ for some constant $C>0$ depending on $t_2$. Let $q_t^+$ be the solution of
\[
\dd q_t^+ = \frac{1}{\ve}(tq_t^+ - (q_t^+)^3 + \ve(1+r(\si)))\, \dd t + \frac{\si}{\sqrt{\ve}}\, \dd W_t \;,\quad q^+_{t_0} = x+3\;.
\]
By Lemma \ref{comp} below, $\eta_t \leq q_t^+$ for all $t \in [t_0,t_2]$ and almost all paths in $(E_1 \cup E_2)^c$. Therefore, $q_t \leq \eta_t - \ve^{\beta}P_t \leq q_t^+ - \ve^{\beta}P_t$, which gives the upper bound.
\end{proof}
This proposition allows us to complete the proof of Theorem \ref{thm:mass sample paths}.
\begin{proof}[Proof of Theorem \ref{thm:mass sample paths}]
Firstly, let us consider the fast pulling case. That is, $\si^{4/3}|\ln \si|^{2/3} \ll \ve \ll 1$. In this case, when $\beta > 2$ the term $1-r(\si)$ appearing in (\ref{eq:qpm}) is strictly positive and bounded away from zero for all $\si$ sufficiently small. Then the analysis of Section \ref{sec:fast} can be applied to $q^-_t$ (the slightly different drift term does not matter). By Lemma \ref{qp}, we can assume $(E_1 \cup E_2)^c$ to hold, in which case $|\ve^{\beta}P_t| \leq \si \ve^{-1/2 +\beta/2 - \delta}$ for all $t \in [-T + 2\ve^{\beta-\delta},t_2]$. Therefore, we find that $q^-_t - \ve^{\beta}P_t > \gamma \sqrt{t}$ for all $c_1 \sqrt{\ve|\ln \si|} \leq t \leq t_2$, where $c_1,\,\gamma > 0$ are suitably chosen constants.

For the slow pulling case, let $0 < \delta < \beta/2-1$ and $\si^{2/(1+2\delta)} \ll \ve \ll \si^{4/3}|\ln \si|^{-13/6}$. Then again $1 \pm r(\si)$ is positive and bounded away from zero for $\si$ small and the analysis in Section \ref{sec:slow} applies to $q^+_t$ and $q^-_t$. Note that, in the limit, $q^-_t$ and $q^+_t$ must ``go the same way'' by comparison of their drift terms and initial positions and we know by Proposition~\ref{prop:over} that the probabilities are $1/2$ in either direction.  As above, the term $\ve^{\beta}P_t$ does not change anything. Therefore, $q_t$ behaves in the same way as $q^-_t$ and $q^+_t$.
\end{proof}

\appendix
\section{A comparison principle for SDE}
We present a lemma that is a slightly modified version of Theorems 5.1 and 5.2 appearing in \cite{And72}. Our proof follows those given there.

Let $(\Omega, \caF,\bbP)$ be a probability space on which is defined a one-dimensional Brownian motion $W$ adapted to a filtration $(\caF_t)_{t \geq 0}$. For $t \geq 0$, let $X(t)$ and $Y(t)$ be two real-valued processes evolving according to
\be \label{eq:XY}
\begin{split}
X(t) & = x_0 + \int_0^t a(s,X(s),Z(s))\, \dd s + C W_t\;,\\
Y(t) & = x_0 + \int_0^t c(s,Y(s))\, \dd s + C W_t\;,
\end{split}
\ee
where $C$ is a constant, $a: [0,\infty) \times \bbR^2 \to \bbR$, $c: [0,\infty) \times \bbR \to \bbR$ are continuous functions, $Z(t)$ is an $\caF_t$-adapted process with continuous sample paths almost surely and $Z(0) = z_0$, which is $\caF_0$-adapted. Note that $X(t)$ and $Y(t)$ have the same initial condition, which may be constant or random as long as it is $\caF_0$-adapted.
\begin{lemma}\label{comp}
Suppose there are constants $C_1,\, C_2 > 0$ such that whenever $|x| < C_1$ and $|z|<C_2$, $a(t,x,z) < c(t,x)$ for all $t \geq 0$. If, almost surely, $|X(t)| < C_1$ and $|Z(t)| < C_2$ for all $t \geq 0$ then 
\[
\bbP\{Y(t) \geq X(t) \text{ for all }t \geq 0\} = 1\;.
\]
\end{lemma}
\begin{proof}
Define $\tau = \inf\{t>0: Y(t) - X(t)<0\}$ and set $\tau = +\infty$ if $Y(t) \geq X(t)$ for all $t \geq 0$. This is a stopping time because if $t>0$ then
\[
\{\tau \geq t\} = \bigcap_{r \in [0,t]\cap \bbQ}\{Y(r) - X(r) \geq 0\} \in \caF_t\;.
\]
Put $D = \{\tau < +\infty\}$ and assume that $\bbP(D)>0$. Then we can define a probability measure $\bbQ(\cdot) = \bbP(\cdot|D)$ on $\caF$. Let
\begin{align*}
X^+(t) & = X(t+\tau)\;,\\
Y^+(t) & = Y(t+\tau)\;.
\end{align*}
By continuity, $\bbQ\{X^+(0) = Y^+(0)\} = 1$. For any $t \geq 0$, we can therefore write (almost surely with respect to $\bbQ$)
\begin{align*}
Y^+(t) - X^+(t) & = Y(t+\tau) - X(t+\tau)\\
& = Y(\tau) - X(\tau) + \int_{\tau}^{t+\tau} \! c(s,Y(s)) - a(s,X(s),Z(s))\, \dd s\\
& = \int_0^t \! c(s+\tau,Y^+(s)) - a(s+\tau,X^+(s),Z^+(s))\, \dd s\;.
\end{align*}
The right-hand side is continuously differentiable in $t$ and so
\[
\bbQ\left\{\lim_{t \to 0}\frac{Y^+(t) - X^+(t)}{t} = c(\tau,X^+(0)) - a(\tau,X^+(0),Z^+(0))\right\} = 1 \;.
\]
Therefore, $\bbQ\{Y^+(t) > X^+(t) \text{ for all sufficiently small } t>0 \} = 1$. But due to continuity of $X$ and $Y$ and the definition of $\tau$, this probability should be zero. This contradiction arises from the assumption that $\bbP(D)>0$. Therefore, $\bbP\{\tau = +\infty\} = 1$.

\end{proof}


\begin{thebibliography}{BEGK04}
\expandafter\ifx\csname url\endcsname\relax
  \def\url#1{\texttt{#1}}\fi
\expandafter\ifx\csname urlprefix\endcsname\relax\def\urlprefix{URL }\fi

\bibitem[AB09]{AB09}
\textsc{M.~Allman} and \textsc{V.~Betz}.
\newblock Breaking the chain.
\newblock \emph{Stochastic Process. Appl.} \textbf{119}, no.~8, (2009),
  2645--2659.

\bibitem[All10]{Allman}
\textsc{M.~Allman}.
\newblock \emph{Chains of interacting Brownian particles under strain}.
\newblock Ph.D. thesis, Warwick University, 2010.

\bibitem[And72]{And72}
\textsc{W.~J. Anderson}.
\newblock Local behaviour of solutions of stochastic integral equations.
\newblock \emph{Trans. Amer. Math. Soc.} \textbf{164}, (1972), 309--321.

\bibitem[BEGK04]{BEGK04}
\textsc{A.~Bovier}, \textsc{M.~Eckhoff}, \textsc{V.~Gayrard}, and
  \textsc{M.~Klein}.
\newblock Metastability in reversible diffusion processes. {I}. {S}harp
  asymptotics for capacities and exit times.
\newblock \emph{J. Eur. Math. Soc. (JEMS)} \textbf{6}, no.~4, (2004), 399--424.

\bibitem[Bel78]{B78}
\textsc{G.~I. Bell}.
\newblock Models for the specific adhesion of cells to cells.
\newblock \emph{Science} \textbf{200}, (1978), 618--627.

\bibitem[BG02]{BG02}
\textsc{N.~Berglund} and \textsc{B.~Gentz}.
\newblock Pathwise description of dynamic pitchfork bifurcations with additive
  noise.
\newblock \emph{Probab. Theory Related Fields} \textbf{122}, no.~3, (2002),
  341--388.

\bibitem[BG06]{BG}
\textsc{N.~Berglund} and \textsc{B.~Gentz}.
\newblock \emph{Noise-induced phenomena in slow-fast dynamical systems}.
\newblock Probability and its Applications (New York). Springer-Verlag London
  Ltd., London, 2006.

\bibitem[DFKU03]{DFKU03}
\textsc{O.~K. Dudko}, \textsc{A.~E. Filippov}, \textsc{J.~Klafter}, and
  \textsc{M.~Urbakh}.
\newblock Beyond the conventional description of dynamic force spectroscopy of
  adhesion bonds.
\newblock \emph{PNAS} \textbf{100}, no.~20, (2003), 11378--11381.

\bibitem[dlm07]{dlmf}
\emph{Digital library of mathematical functions}.
\newblock National Institute of Standards and Technology, 2010-05-07.
\newblock \urlprefix\url{http://dlmf.nist.gov/}.

\bibitem[ER97]{ER97}
\textsc{E.~Evans} and \textsc{K.~Ritchie}.
\newblock Dynamic strength of molecular adhesion bonds.
\newblock \emph{Biophysical Journal} \textbf{72}, no.~4, (1997), 1541 -- 1555.

\bibitem[Eyr35]{Eyr35}
\textsc{H.~Eyring}.
\newblock The activated complex in chemical reactions.
\newblock \emph{J. Chem. Phys.} \textbf{3}, (1935), 107--115.

\bibitem[Fri08]{Fri08}
\textsc{R.~W. Friddle}.
\newblock Unified model of dynamic forced barrier crossing in single molecules.
\newblock \emph{Phys. Rev. Lett.} \textbf{100}, no.~13, (2008), 138302.

\bibitem[FS09a]{FS09b}
\textsc{S.~Fugmann} and \textsc{I.~M. Sokolov}.
\newblock Non-monotonic dependence of the polymer rupture force on molecule
  chain length.
\newblock \emph{EPL} \textbf{86}, no. 28001, (2009), 1--5.

\bibitem[FS09b]{FS09a}
\textsc{S.~Fugmann} and \textsc{I.~M. Sokolov}.
\newblock Scaling of the rupture dynamics of polymer chains pulled at one end
  at a constant rate.
\newblock \emph{Phys. Rev. E} \textbf{79}, no.~2, (2009), 021803.

\bibitem[FW98]{FW98}
\textsc{M.~I. Freidlin} and \textsc{A.~D. Wentzell}.
\newblock \emph{Random perturbations of dynamical systems}, vol. 260.
\newblock Springer-Verlag, New York, second ed., 1998.

\bibitem[Kra40]{K40}
\textsc{H.~A. Kramers}.
\newblock Brownian motion in a field of force and the diffusion model of
  chemical reactions.
\newblock \emph{Physica (Utrecht)} \textbf{7}, (1940), 284--304.

\bibitem[Lan69]{Lan69}
\textsc{J.~S. Langer}.
\newblock Statistical theory of the decay of metastable states.
\newblock \emph{Ann. Phys.} \textbf{54}, no.~2, (1969), 258 -- 275.

\bibitem[LCST07]{LCST07}
\textsc{H.-J. Lin}, \textsc{H.-Y. Chen}, \textsc{Y.-J. Sheng}, and
  \textsc{H.-K. Tsao}.
\newblock Bell's expression and the generalized {Garg} form for forced
  dissociation of a biomolecular complex.
\newblock \emph{Phys. Rev. Lett.} \textbf{98}, no.~8, (2007), 088304.

\bibitem[Lee09]{Lee09}
\textsc{C.~F. Lee}.
\newblock Thermal breakage of a discrete one-dimensional string.
\newblock \emph{Phys. Rev. E} \textbf{80}, no.~3, (2009), 031134.

\bibitem[REB{\etalchar{+}}06]{REBENMRAR06}
\textsc{M.~Raible}, \textsc{M.~Evstigneev}, \textsc{F.~W. Bartels},
  \textsc{R.~Eckel}, \textsc{M.~Nguyen-Duong}, \textsc{R.~Merkel},
  \textsc{R.~Ros}, \textsc{D.~Anselmetti}, and \textsc{P.~Reimann}.
\newblock Theoretical analysis of single-molecule force spectroscopy
  experiments: Heterogeneity of chemical bonds.
\newblock \emph{Biophys. J.} \textbf{90}, (2006), 3851--3864.

\bibitem[SDG06]{SDG06}
\textsc{A.~Sain}, \textsc{C.~L. Dias}, and \textsc{M.~Grant}.
\newblock Rupture of an extended object: A many-body {Kramers} calculation.
\newblock \emph{Phys. Rev. E} \textbf{74}, no.~4, (2006), 046111.

\end{thebibliography}

\newcommand{\etalchar}[1]{$^{#1}$}

\end{document}